\documentclass[11pt]{amsart}
\usepackage[utf8]{inputenc}
\usepackage{amsmath, amssymb, amsthm,faktor}
\usepackage{enumerate}
\usepackage[curve]{xypic}
\usepackage{graphicx}
\usepackage[usenames]{xcolor}
\usepackage{graphicx}
\usepackage{ulem}

\usepackage{color} 

\newcounter{warn}[page]
\newcommand{\danger}{${\color{red}\triangle}\llap{\raisebox{.3ex}%
{\tiny!\hspace{1.45ex}}}$}
\newcommand{\warning}[1]{%
\raisebox{.01em}[0em]{\danger\ifnum\value{warn} > 1%
\tiny\bf\arabic{warn}\fi}%
\marginpar{\color{red}\tiny{\ifnum\value{warn} > 1\tiny\bf\arabic{warn}:\fi}\tiny #1}%
\stepcounter{warn}}

\newtheorem{theorem}{Theorem}[section]
\newtheorem{definition}[theorem]{Definition}

\newtheorem{proposition}[theorem]{Proposition}
\newtheorem{lemma}[theorem]{Lemma}

\theoremstyle{remark}
\newtheorem{remark}{Remark}

\renewcommand{\emptyset}{\varnothing}

\newcommand{\N}{\mathbb{N}}
\newcommand{\Z}{\mathbb{Z}}

\newcommand{\R}{\mathbb{R}}
\newcommand{\C}{\mathbb{C}}
\newcommand{\CP}{\mathbb{CP}}
\renewcommand{\S}{\mathbb{S}}

\newcommand{\Id}{\mathrm{Id}}

\newcommand{\eval}{\mathrm{ev}}
\newcommand{\cutoff}{\chi}
\newcommand{\Steps}{{\mathcal{S}}}
\newcommand{\Gen}[2][\star]{{\mathcal{L}(#2,#1)}}
\newcommand{\GenBrut}[2][\star]{{\tilde{\mathcal{L}}(#2,#1)}}
\newcommand{\Rel}[2][\star]{{\mathcal{R}(#2,#1)}}
\newcommand{\RelHomGen}{R_{\mathrm{FMF}}}
\newcommand{\RelMorseGen}{R_{\mathrm{M}2}}
\newcommand{\Orbits}{\mathcal{P}}
\newcommand{\CovOrbits}{\tilde{\Orbits}}
\newcommand{\Crit}{\mathrm{Crit}}
\newcommand{\U}{\mathcal{U}}
\newcommand{\M}{\mathcal{M}}
\newcommand{\Mint}{\overset{\ \circ}{\M}}
\newcommand{\MMorse}{\M_{\mathrm{Morse}}}
\newcommand{\MintMorse}{\Mint_{\mathrm{Morse}}}
\newcommand{\glueto}{\overset{\sharp}{\rightarrow}}

\newcommand{\orient}{\epsilon}

\newcommand{\Cardabs}{\mathop{\sharp_{\mathrm{abs}}}}
\newcommand{\Cardalg}{\mathop{\sharp_{\mathrm{alg}}}}

\begin{document}

\title[A Floer Fundamental group]%
{A Floer Fundamental group
\\[1ex] 
\sc \small Groupe fondamental de Floer%
}%

\author[J.-F. Barraud]{J.-F. Barraud}%
\address{%
  Institut de mathématiques de Toulouse\\
  Université Paul Sabatier -- Toulouse III\\
  118 route de Narbonne\\
  F-31062 Toulouse Cedex 9\\
  FRANCE} %

\begin{abstract}
 The main purpose of this paper is to provide a description of the
 fundamental group of a symplectic manifold in terms of Floer theoretic objects. 
 As an application, we show that when counted with a suitable notion of
 multiplicity, non degenerate Hamiltonian diffeomorphisms have enough
 fixed points to generate the fundamental group. 

 \bigskip

 \noindent \textsc{Résumé.} 
 L'objet de cet article est de donner une description du groupe fondamental d'une
 variété symplectique en terme d'objets de la théorie de Floer.
 A titre d'application, on montre que tout difféomorphisme Hamiltonien non dégénéré a,
 si on les compte avec une notion convenable de multiplicité, suffisamment 
 de points fixes pour engendrer le groupe fondamental.
\end{abstract}

\maketitle

\bigskip

\section{Introduction}
\subsection{Presentation of the results}
In many ways, the topology of a space influences its geometry, and this
is particularly true in symplectic geometry. Having a symplectic
interpretation of a topological invariant is a good tool to explore this
relationship. The celebrated Floer Homology (\cite{Floer3}\cite{Floer4})
is of course a strong illustration of this phenomenon. Introduced to
prove the homological version of the Arnold conjecture (\cite{Arnold}),
it quickly became one of the most powerful tools in symplectic geometry.

However, all the techniques derived from the original Floer construction
are homological, or at least chain complex based in nature. The notion of
cobordism (among moduli spaces) is even at the root of the original ideas
of M.~Gromov \cite{Gromov} of using pseudo-holomorphic curves to derive
invariants in symplectic geometry. The use of local coefficients in Floer
complexes allows Floer theory to involve some homotopical invariants, but
purely homotopical tools are still missing, and it is the goal of this
paper to provide a Floer theoretic interpretation of the fundamental
group.

All the objects this construction is based on are still classical Floer
theoretic objects, but the essential non Abelian phenomena that make the
difference between the fundamental group and the first homology
group are caught by a deeper use of $1$-dimensional moduli
spaces, and the use of ``augmentations''.

More precisely, let $(M,\omega)$ be a connected closed monotone
symplectic manifold and choose an Hamiltonian function $H$ on $M$, a
possibly time dependent almost complex structure $J$ compatible with
$\omega$, and a point $\star$ in $M$ to serve as the base point. Recall
the Floer trajectories in this setting are (finite energy) maps
$u:\R\times S^{1}\to M$ satisfying the Floer equation
$$
\frac{\partial u}{\partial s}(s,t) +
J_{t}(u(s,t))\frac{\partial u}{\partial t}(s,t)
 =
 J_{t}(u(s,t)) X_{H_{t}}(u(s,t)),
$$
where $X_{H}$ is the Hamiltonian vector field associated to $H$. 

Using a cutoff function $\chi$ to turn off the non homogeneous
Hamiltonian term on the positive end of the tube (resp. on both ends but
preserving it on an annulus of varying modulus) allows to define moduli
spaces denoted by $\M(x,\emptyset)$ (resp. $\M(\star,\emptyset)$), which
are Floer counterparts of Morse unstable manifolds (see
the comments after definition \ref{def:FloerLoopStep}). %
It is a classical result of Floer theory that for a generic set of
auxiliary data $(H,J,\star,\chi)$, all these moduli spaces are smooth
finite dimensional manifolds.

Similarly to the Morse setting where a loop can be seen as a
concatenation of paths associated to unstable manifolds of index $1$
critical points, we use the components of the above $1$-dimensional
moduli spaces to define a notion of Floer loop (see definition
\ref{def:FloerLoops}). These loops come naturally with concatenation and
cancellation relations for which they form a group $\Gen{H}$. The main
statement of the paper is then the following theorem~:
\begin{theorem}\label{thm:RoughSurjectivity}
  There is a natural evaluation map that induces a surjective group
  homomorphism $\xymatrix{\Gen{H,J}\ar@{>>}[r]&\pi_{1}(M,\star)}$.
\end{theorem}

A description of the relations is also given, but, although they
obviously only depend on $H$, $J$, $\star$, and $\chi$, we resort to an
auxiliary Morse function to get a finite presentation for them (see
section \ref{sec:Relations}). Nevertheless, we produce explicit relations
such that the generated normal subgroup $\Rel{H}$ satisfies the following
statement~:
\begin{theorem}\label{thm:RoughIsomorphism}
The evaluation map induces a group isomorphism
$$
\xymatrix{\faktor{\Gen{H}}{\Rel{H}}\ar^{\ \ \ \ \sim}[r]& \pi_{1}(M,\star)}.
$$  
\end{theorem}

\smallskip

Notice the construction is presented here in the absolute setting, i.e.
Hamiltonian fixed points problem, but also makes sense in the relative
one, i.e. intersections of a Lagrangian sub-manifold with its
deformations under Hamiltonian isotopies problem. Although the latter can
be expected to hold the most interesting applications, we choose to focus
on the former for the sake of simplicity and to better highlight the main
ideas~: the generalization to the latter entails exactly the same issues
as for the homology and involves no new idea. 

Finally, the construction also makes sense in the stable Morse setting
(i.e. study of Morse functions  that are quadratic at infinity on
$M\times\R^{N}$). Although the corresponding results have their own
interest and would deserve a separate discussion, they will only be
quickly sketched without proofs in the last section of this paper (see
section \ref{sec:StableMorse}), rather as an illustration and a
simplified finite dimensional model of the Floer setting.

\medskip

A natural outcome of this construction is an estimate on the number of
fixed points of Hamiltonian diffeomorphisms, but not in usual way, since
a notion of multiplicity has to be introduced. Indeed, rather than the
critical points themselves, the relevant objects required to build loops
are their unstable manifolds (called ``steps'' in the sequel), and while
to one critical point corresponds exactly one unstable manifold in the
Morse setting, Floer counterparts of unstable manifolds may have several
components, which have all to be taken into account.

Counting the number $\nu_{J}(x)$ (resp. $\nu_{J}(\star)$) of steps
through a given Conley-Zehnder index $1-n$ fixed point $x$ (resp.
$\star$) defines a notion of multiplicity for these points (that depends
on the almost complex structure, see definition \ref{def:Multiplicity}
for more details). We then  have the following theorem~:

\begin{theorem}\label{thm:RoughEnoughPoints}
 Let $\rho(\pi_{1}(M))$ denote the minimal number of generators of the
 fundamental group. Then~:
  \begin{equation}
    \label{eq:RoughEnoughPoints} \nu_{J}(\star) + \sum_{|y|=1}\nu_{J}(y)
    \geq \delta(\pi_{1}(M)),
  \end{equation}
  where the sum runs over the contractible $1$-periodic orbits, or more
  precisely over the homotopy classes of cappings of such orbits with
  Conley-Zehnder index $1-n$. 
\end{theorem}

  \begin{remark}   
    The number $\nu_{J}(\star)$ is a sum of contributions of index $-n$
    fixed points (see definition \ref{def:Multiplicity}), so that
    inequality \eqref{eq:RoughEnoughPoints} can be interpreted as a lower
    bound for the number of particular Floer configurations associated
    to fixed points with Conley-Zehnder index $-n$ and $1-n$.
  \end{remark}

\begin{remark}
  This statement should be compared to its Morse analogue, namely that
  for any Morse function $f:M\to\R$, we have
  \begin{equation}
    \label{eq:MorseEnoughPoints}
    \sharp \Crit_{1}(f)\geq\rho(\pi_{1}(M)),
  \end{equation}
  where $\Crit_{1}(f)$ denotes the set of index $1$ critical points.

  As already mentioned, the construction, and hence the definition of the
  multiplicities makes sense in the stable Morse and a fortiori Morse
  settings (moreover, we claim, without proof, that for $C^{2}$ small
  Morse functions, Morse and Floer moduli spaces can be identified like
  in \cite{Floer3}, and that multiplicities coïncide in this case).  For
  an index $1$ Morse critical point $y$, $\nu_{J}(y)$ is the number of
  components of its unstable manifold, and hence always evaluates to $1$.
  Similarly, $\nu_{J}(\star)+1$ is the number of Morse trajectories
  through $\star$ and hence evaluates to $1$. As a consequence,
  $\nu_{J}(\star)=0$, so that
%
  $$
  \nu_{J}(\star) + \sum_{|y|=1}\nu_{J}(y) = \sharp \Crit_{1}(f),
  $$
  and \eqref{eq:RoughEnoughPoints} appears as a generalization of
  \eqref{eq:MorseEnoughPoints} to the more general Floer setting.
\end{remark}

\begin{remark}
  There is no hope to avoid multiplicities in \eqref{eq:RoughEnoughPoints} as
  long as it results from a construction that also applies to the stable
  Morse setting, which is the case of ours.

  Indeed, M. Damian showed in \cite{Mihai} that the stable Morse number
  (which is the minimal number of critical points of a Morse function
  which is quadratic at infinity on a product $M\times \R^{N}$) may be
  strictly smaller than the Morse number (which is the minimal number of
  critical points of a Morse function on $M$), and that stable Morse
  functions may not have enough points to generate the fundamental group
  (for instance, such functions do exist on manifolds whose fundamental
  group is $(A_{5})^{20}$).

  This implies that the multiplicities in \eqref{eq:RoughEnoughPoints}
  are mandatory, and the construction offers a new point of view on this
  question~: although there may not be enough geometric critical points
  to generate the fundamental group, it explains how the same point can
  define several generators to overcome this deficit and still recover
  the fundamental group.
\end{remark}


\begin{remark}
  The inequality \eqref{eq:RoughEnoughPoints} is obviously different in
  nature from the Morse inequalities derived from the Floer homology,
  since one may have $\rho(\pi_{1}(M))>\beta_{1}(M)$ (where
  $\beta_{1}(M)$ is the first Betti number of $M$). It is also different
  from the results of K.~Ono and A.~Pajitnov (\cite{OnoPajitnov}, see
  below) and more generally from any result based on the algebraic study
  of a chain complex that would also apply to the stable Morse setting.
  Indeed, examples are known of stable Morse functions that have
  strictly less critical points than the minimal number of generators of
  the fundamental group. 

%
\end{remark}

\medskip

The role and the control of the contributions of the multiplicities in
general is a deep and intriguing question, closely related to the
estimation of the minimal number of periodic orbits.

The following theorem ensures the existence of at least one Hamiltonian
periodic orbit  with Conley-Zehnder index $1-n$  and  non
vanishing multiplicity provided the fundamental group
  is non trivial~:%

\begin{theorem}\label{thm:RoughAtLeastOneOrbit}
  Let $(M,\omega)$ be a monotone symplectic manifold. Suppose
  $\pi_{1}(M)\neq\{1\}$. Then every non degenerate Hamiltonian function
  has to have at least one contractible $1$-periodic orbit of Conley
  Zehnder index $1-n$. Moreover, for a generic choice of possibly time
  dependent almost complex structure, at least one such orbit has non
  vanishing multiplicity.
\end{theorem}

In particular, this result provides at least one index $1-n$ orbit even
if the first homology group of the manifold vanishes, provided the
fundamental group is non trivial.

One interesting feature of this theorem is that its proof is essentially
geometric, where the usual Floer technics are rather algebraic~: it comes
down to patching suspensions of $1$-dimensional moduli spaces side to
side to form a disc. In this sense, although theorem
\ref{thm:RoughAtLeastOneOrbit} is not strictly speaking a corollary of
the Floer interpretation of the fundamental group given in this paper, it
derives from the same principal idea, namely that $1$-dimensional moduli
spaces do contain information that the homology does not catch.

Moreover, the orbit  exhibited in this statement 
has explicitly non vanishing multiplicity, while this is not immediately
obvious in other constructions that provide lower bounds on the number of
periodic orbits.


\subsection*{Relation to the Arnold conjecture and other results}

Theorem \ref{thm:RoughEnoughPoints} is obviously a variation on the
Arnold conjecture. In its non degenerate and strongest form, this
conjecture claims that the total number of $1$-periodic orbits of a non
degenerate Hamiltonian flow can not be less than the minimal number of
critical points for a Morse function (or stable Morse function in a
weaker form of the conjecture). A weaker but maybe more convincing and
tractable version involves the \emph{stable} Morse number, which is the
minimal number of critical points of Morse functions which are quadratic
at infinity on products $M\times \R^{N}$.

This conjecture is closely related to the birth of symplectic geometry
itself. A strong breakthrough was achieved by A. Floer who constructed
his chain complex to establish the Homological version of the Arnold
conjecture for compact monotone symplectic manifolds, opening the way to
huge efforts by many authors to generalize his original
construction. 

Until very recently however, work regarding this
conjecture was focused on its homological version.

In a recent work \cite{OnoPajitnov}, K.~Ono and A.~Pajitnov use the Floer
complex with local coefficients to extend these constraints to the
Hamiltonian setting. In particular, they show the following
\begin{theorem}[K.~Ono, A.~Pajitnov]\label{thm:OnoPajitnov}
  Suppose $M$ is a weakly monotone symplectic manifold and let $H$ be a
  Hamiltonian function on it. Then, if they are all non degenerate, the
  number $p(H)$ of fixed points of the associated Hamiltonian
  diffeomorphism satisfies
  $$
  \begin{cases}
  p(H)\geq 1 & \text{if }|\pi_{1}(M)|=+\infty\\
  p(H)\geq \delta(\pi_{1}(M))&  \text{if } |\pi_{1}(M)|<+\infty
  \end{cases},
  $$
  where $\delta(\pi_{1}(M))$ is the minimal number of generators of the
  kernel of the augmentation $\Z[\pi_{1}(M)]\to\Z$.
\end{theorem}

\medskip

Similarly to the stable Morse setting, the points of view of this theorem
and theorem \ref{thm:RoughEnoughPoints} are essentially different~: the
former focuses on the number of geometric fixed points, while the latter
associates possibly several generators to the same geometric orbit to
overcome an eventual lack of generators and still recover the fundamental
group. 


\subsection{Organization of the paper}

In the second section of the paper (the first is this introduction), the
main definitions, statements and technical tools are presented. The third
section is dedicated to the comparison of Morse and Floer loops, and the
proof of theorem \ref{thm:RoughSurjectivity}. The fourth section is
devoted to the description of the relations, and the fifth to the proof
of the application (theorem \ref{thm:RoughEnoughPoints}) and theorem
\ref{thm:RoughAtLeastOneOrbit}. Finally, the last section is a
  sketch without proofs of the construction in the Stable Morse setting.

\bigskip

This work would not exist without the crucial help of a few people. I am
particularly thankful to O.~Cornea, whose deep topological insight and
generosity nourished me for years, to J.-Y.~Welschinger and B.~Chantraine
to whom I am indebted for the keystone of this paper, which is the notion
of augmentation, to A.~Oancea who served as a compass to me and to
M.~Damian who also owns a large part of this work. Finally, I'm
particularly grateful to A.V. Duffrène who indirectly but deeply
influenced the birth of this paper.

\section{Main definitions and  statements.}\label{sec:MainStatements}

\subsection{Preliminaries}\label{sec:Preliminaries}

Let $(M,\omega)$ be a $2n$-dimensional connected compact symplectic
manifold without boundary. For technical reasons, $M$ will be supposed to
be either 
\begin{itemize}
\item 
  symplectically aspherical, which means $\omega$ vanishes on the image
  of the Hurewicz homomorphism $\pi_{2}(M)\to H_{2}(M)$, or
\item 
  monotone, which means $c_{1}(TM)$ and $\omega$ are proportional by a
  positive constant on the image of the Hurewicz homomorphism
  $\pi_{2}(M)\to H_{2}(M)$.
\end{itemize}

These assumptions will allow us to easily
\begin{itemize}
\item 
  avoid the transversality issues related to the multiply covered
  negative curves,
\item
  avoid bubbles on $0$ and $1$-dimensional moduli spaces,
\item
  ensure finiteness of the number of (lifted) orbits of given
  Conley-Zehnder index.
\end{itemize}

\smallskip

Given a Hamiltonian function $H:M\times\S^{1}\to\R$, we let $X_{H}$ be
the associated Hamiltonian vector field, $\phi^{t}_{H}$ its flow, and
$\Orbits(H)$ the set of its contractible $1$-periodic orbits.

To handle the index computation when $c_{1}(TM)$ does not vanish on
$\pi_{2}(M)$, we consider the covering $\CovOrbits(H)$ associated to the
group $\pi_{2}(M)/\ker c_{1}$. It is obtained from $\Orbits(H)$ by
adjoining a capping class to the orbit in the
following way~:
\begin{equation}
  \label{eq:CovOrbits}
\CovOrbits(H) = \faktor{\{(\gamma,\bar\gamma), 
\gamma\in\Orbits(H), \bar\gamma:D\to M,
\bar\gamma_{|\partial D}=\gamma\}}{\sim}
\end{equation}
where $D$ is the $2$-disc and
$(\gamma,\bar\gamma)\sim(\gamma',\bar\gamma')$ iff $\gamma=\gamma'$ and
$\mu_{CZ}(\bar\gamma)=\mu_{CZ}(\bar\gamma')$ ($\mu_{CZ}$ denoting the
Conley-Zehnder index). 
  
Notice that this last equality implies that the two cappings also have
the same symplectic area~: glued along their boundary, $\bar\gamma$ and
$\bar\gamma'$ with the reversed orientation form a sphere $S$ with
vanishing first Chern class, and because of our asphericity or
monotonicity assumption, $\iint_{S}\omega=0$, which means that
$\bar\gamma$ and $\bar\gamma'$ have the same symplectic area. As a
consequence, both the Conley-Zehnder index and the symplectic area are
well defined for equivalence classes of cappings.

In the sequel, $\CovOrbits(H)$ will completely replace $\Orbits(H)$ and
no explicit reference to the covering will be made anymore. In
particular, what we call a Hamiltonian orbit from now on will in fact be
a lift of such an orbit to $\CovOrbits(H)$.

Each element $x$ in $\CovOrbits(H)$ then has a well defined
Conley-Zehnder index $\mu_{CZ}$.
For convenience, we shift the Conley-Zehnder index by $n$ and let
$$
|x| = \mu_{CZ}(x) + n.
$$
The set of orbits $\CovOrbits(H)$ splits according to this index, and we let
$$
\CovOrbits_{k}(H) = \{x\in\CovOrbits(H),\ |x|=k\}.
$$

\bigskip

Given a (possibly time dependent) $\omega$-compatible almost complex
structure $J$, we are interested in the Floer moduli spaces and some
classical variants of such that we describe below. Recall the Floer
equation for a map $u:\R\times\S^{1}\to M$ is the following~:
\begin{equation}
  \label{eq:Floer}
  \frac{\partial u}{\partial s}+J_{t}(u)\Big(\frac{\partial u}{\partial t}
  -X_{H}(t,u)\Big)=0.
\end{equation}
Moreover, we fix once for all a
smooth function $\cutoff: \R\to[0,1]$ such that
$$
\begin{cases}
\cutoff(s)=1 & \text{ if } s\leq-1\\
\cutoff(s)=0 & \text{ if } s\geq 0
\end{cases},
$$
and use it to cutoff the Hamiltonian term of the Floer equation on one or
both ends of the cylinder by considering the equation 
\begin{equation}
  \label{eq:StripedFloerEquation}
\tag{$F_{i}$}
\frac{\partial u}{\partial s}+J_{\chi_{i}(s)t}(u)\Big(\frac{\partial u}{\partial t}
-\chi_{i}(s)\ X_{H}(t,u)\Big)=0
\end{equation}
for different functions $\chi_{i}:\R\to[0,1]$, $i=1,2,3$ and
  $i=(4,R)$, derived from $\cutoff$, namely (see figure
\ref{fig:CutOffFunctions})
\begin{enumerate}
\item \label{it:Floer}
$\chi_{1}\equiv 1$ defines the usual Floer equation, 
\item\label{it:capoff}
$\chi_{2}(s)=\cutoff(s)$ defines the ``lower capping equation'', 
\item\label{it:capon}
$\chi_{3}(s)=\cutoff(-s)$ defines the ``upper capping equation'', 
\item\label{it:sphere}
$\chi_{4,R}(s)=\cutoff(s-R)\cutoff(-s-R)$ defines ``$R$-perturbed sphere
equation''. 
\end{enumerate}

\begin{figure}[!ht]
  \centering
  \includegraphics[width=10cm]{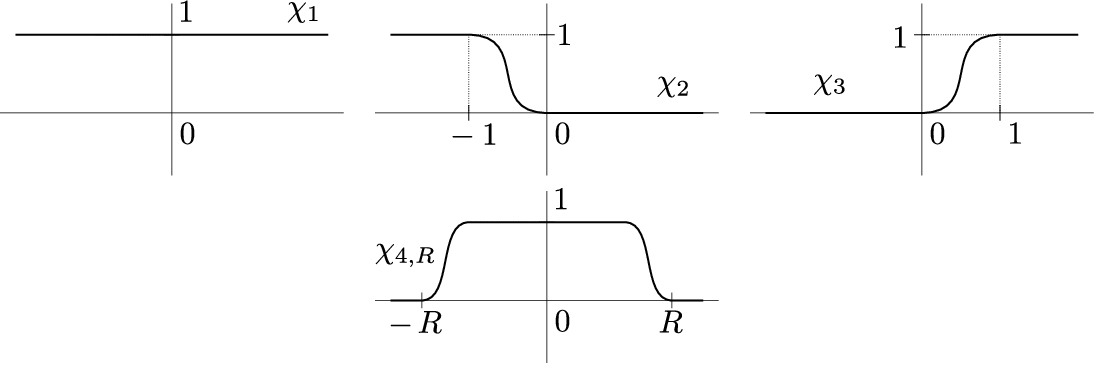}
  \caption{Cutoff functions}
  \label{fig:CutOffFunctions}
\end{figure}

In ($F_{4,R}$), $R$ is a real parameter, but notice that for
$R\leq 0$, the equation has no Hamiltonian term anymore and does not
depend on $R$~: $R$ will hence be considered in $[0,+\infty)$.

Recall that the energy of a solution $u$ of this equation is defined as
$$
 E(u) = \iint\omega(\frac{\partial u}{\partial s}, 
                    \frac{\partial u}{\partial t}-\chi_{i}(s)X_{H}(u))dsdt
 = \iint\Vert \frac{\partial u}{\partial s} \Vert^{2}_{\chi_{i}(s)t}dsdt,
$$
where $\Vert.\Vert_{t} = \omega(.,J_{t}.)$. Solutions of finite energy of
this equation have converging ends, either to a point by the classical
removal of singularities argument (\cite{AudinLafontaine},\cite{McDS}) if
the Hamiltonian term is cut off on this end, or to a Hamiltonian orbit if
not (\cite{Floer5}). In the former case, considering the end as a
neighborhood of $0$ in $\C\setminus\{0\}$, the map $u$ extends
holomorphically through $0$, and the equations $(F_{i})$ above could
equivalently be considered as defined on the sphere $\CP^{1}$ with $2$,
$1$ or no puncture (see for instance \cite{OhZhu} for a more uniform
description of these equation, or \cite{McDS} chapter 8 for the case
without punctures, i.e. equation $(F_{4,R})$ with fixed $R$). Anyway, on
an end where the Hamiltonian term is cut off, the limit value will be
denoted by $u(+\infty)$ or $u(-\infty)$. We abusively but conveniently
write that such a trajectory ends at the $\emptyset$ symbol to describe
the fact that this limit point is not constrained.

We are interested in the moduli spaces described below and depicted on
figure \ref{fig:FloerModuliSpaces}. Let $\star$ be a point in $M$, and
$\U$ be the space of smooth maps
$u:\R\times\S^{1}\to M$ that have finite energy i.e. such that
$\iint\Vert\frac{\partial u}{\partial s}\Vert^{2}dsdt<+\infty$. If $a$ is
an oriented disc, let $\bar{a}$ denote the disc with opposite
orientation, and if $b$ is another disc or tube having the same boundary
as $a$ , let $\bar
  a\sharp b$ denote the gluing of the two. 
\begin{align}
  \Mint(y,x)&=\faktor{\{u\in\U, (F_{1}),
\substack{\lim_{s\to-\infty}u(s,\cdot)=y\\\lim_{s\to+\infty}u(s,\cdot)=x}, 
   \text{ and } [y\sharp u\sharp\bar{x}] = 0\}}{\R}\\
  \Mint(y,\emptyset)&=\{u\in\U, (F_{2}), \lim_{s\to-\infty}u(s,\cdot)=y, 
   \text{ and } [y\sharp u] = 0\}\\
  \Mint(\star,x)&=\{u\in\U, (F_{3}), 
 \substack{\lim_{s\to-\infty}u(s,\cdot)=\star\\\lim_{s\to+\infty}u(s,\cdot)=y}
\text{ and } [u\sharp\bar{x}] =0\}\\
  \Mint(\star,\emptyset)&=\{(R,u)\in[0,+\infty)\times\U, (F_{4,R}), \lim_{s\to-\infty}u(s,\cdot)=
    \star \text{ and } [u] = 0\}
\end{align}
where the brackets denote classes in $\pi_{2}(M)/\ker c_{1}$, and their
vanishing expresses the compatibility of the trajectory $u$ with the
prescribed lifts of its ends to the covering space $\CovOrbits(H)$.

Notice that in the last case, the parameter $R$ is allowed to vary, and that
the moduli space $\Mint(\star,\emptyset)$ is endowed with the map
$\Mint(\star,\emptyset)\xrightarrow{\pi}[0,+\infty)$ given by
$\pi(u,R)=R$.
\begin{figure}[!ht]
  \centering
  \includegraphics[scale=.6]{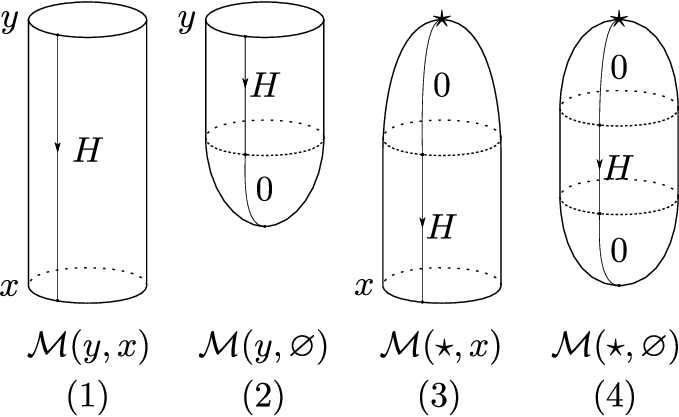}
  \caption{Floer moduli spaces.}
  \label{fig:FloerModuliSpaces}
\end{figure}

The three last types of moduli spaces are used in \cite{PSS} (in
conjunction with a Morse function that we do not use here) to define the
PSS homomorphisms and compare Morse and Floer homologies.

Since the elements of $\Mint(x,\emptyset)$ for $x\in\CovOrbits_{0}(H)$
can be used to define an augmentation on the
Floer complex, we use the following terminology~:
\begin{definition}
  Given an index $0$ Hamiltonian periodic orbit $x\in\CovOrbits_{0}(H)$,
  a capping $\alpha\in\M(x,\emptyset)$ is called an ``augmentation'' of
  $x$, and the couple $(x,\alpha)$ an augmented orbit.  
\end{definition}

It is well known (see remark \ref{rk:StarIsRegular} below) that for a
generic choice of $(H,J,\star,\cutoff)$, these moduli
spaces  are smooth manifolds whose dimension
is prescribed by the end constraints and the
  homotopy class of the tube. 
%
%

\begin{remark}\label{rk:StarIsRegular}
  The transversality issues for the three first moduli spaces are
  discussed in \cite{FHS}.  The last moduli space
  $\Mint(\star,\emptyset)$ is somewhat special with this respect, one
  reason being that for $R=0$, it involves constant maps, for which the
  key argument of being ``somewhere injective'' fails. Transversality for
  the constant maps is particularly relevant to us since it implies that
  such curves are regular for the projection $\M(\star,\emptyset)\to\R$,
  which in turn means that they can be locally ``followed'' as $R$ varies
  (see proposition \ref{prop:FloerLoopStepsStar}). 

  The following proposition ensures that constant spheres are indeed
  regular (for any almost complex structure).
  \begin{proposition}\label{prop:StarIsRegular}
    Recall the projection
    $\Mint(\star,\emptyset)\xrightarrow{\pi}[0,+\infty)$. For $R=0$,
    $\pi^{-1}(R)$  consists in the single point $(u_{\star},0)$ where
    $u_{\star}$ is the constant map at $\star$. This solution is regular,
    which means that (in the suitable functional spaces) the equation
    ($F_{4,0}$) is a submersion at this point. In particular, $0$ is a
    regular value of $\pi$.
  \end{proposition}

\begin{proof}[Sketch of proof]
  Glossing over the definition of the functional spaces in use, observe
  that the problem can be reformulated in terms of maps from $\CP^{1}$ to
  $M$ in the trivial homology class. For $R=0$, equation $(F_{4,R})$ 
  simply becomes
  \begin{gather}
  Du + J_{0}(u)\, Du\, i = 0\label{eq:J-holo}.
  \end{gather}
  Points of $\M(\star,\emptyset)$ lying above $R=0$ are hence
  $J_{0}$-holomorphic spheres in the trivial homology class and are
  therefore constant. The additional condition $u(0)=\star$ implies
  $\pi^{-1}(0)=\{(u_{\star},0)\}$.

  The linearization (with respect to $u$) of the left hand term in
  \eqref{eq:J-holo} at the constant map $u_{\star}$ leads to a linear
  operator $L$ defined for maps from $\CP^{1}$ to a fixed
  $\C^{n}=T_{\star}M$ of the form
  \begin{gather}
     L\dot{u} = D\dot{u} +  J_{0} D\dot{u}\, i, \label{eq:LinFloer}
  \end{gather}
  where $J_{0}=J_{0}(\star)$ is constant. The kernel of $L$ consists of
  the holomorphic spheres in $\C^{n}$ and hence of the constants. It is
  therefore $2n$-dimensional and since $2n$ is also the index of $L$, this
  implies that $L$ is surjective, which easily implies the required
  submersion property.
\end{proof}
\end{remark}

In particular, under a generic choice of $(H,J,\star,\cutoff)$, we have ~:
\begin{align*}
  \dim\Mint(y,x)&=|y|-|x|-1 \\ 
  \dim\Mint(y,\emptyset)&=|y|\\  
  \dim\Mint(\star,x)&=-|x|\\  
  \dim\Mint(\star,\emptyset)&=1
\end{align*}
From now on, $(H,J,\star,\cutoff)$ will be supposed to be chosen so
that all these moduli spaces are indeed cut out
transversely.

Moreover, all these moduli spaces are compact up to breaking or bubbling
off of spheres, and we let
\begin{align*}
  \M(x,y)            &=\overline{\Mint(x,y)            }& 
  \M(x,\emptyset)    &=\overline{\Mint(x,\emptyset)    }\\  
  \M(\star,y)        &=\overline{\Mint(\star,y)        }&  
  \M(\star,\emptyset)&=\overline{\Mint(\star,\emptyset)}
\end{align*}
be the Gromov-Floer compactifications of the previous moduli spaces.

\begin{remark}
  Notice however that $\M(\star,\emptyset)$ has a ``built-in'' (i.e.
  already present in $\Mint(\star,\emptyset)$) boundary component, , that
  does not come from the Gromov compactification but from the limit case
  $R=0$.
\end{remark}

In all this paper, only $0$ and $1$-dimensional moduli spaces will be
considered, and no bubbling of sphere can occur on such moduli spaces.
This means they will all be compact up to breaking and smooth. 

In particular, each $0$-dimensional moduli space $\M(y,x)$ is compact, and
hence finite, and we let 
$$
\Cardabs(\M(y,x))=\sum_{\gamma\in\M(y,x)}(+1)
$$
denote the cardinality of
$M(y,x)$. 
\begin{remark}
 It is usual, when working with pseudo-holomorphic curves or Floer
 trajectories, to consider the algebraic number $\Cardalg\M(x,y)$ of
 elements in a $0$-dimensional moduli space, i.e. to take signs coming
 from some orientation of the moduli space into account. We stress
 however that this definition refers to the absolute number, i.e. the sum
 where each element counts for $+1$.
\end{remark}

\subsection{Floer steps and loops}

Given a configuration of two consecutive isolated Floer trajectories
$(\beta,\alpha) \in \M(y,x) \times \M(x,\emptyset)$ with $x \in
\CovOrbits_{0}(H)$ and $y \in \CovOrbits_{1}(H) \cup \{\star\}$, the
gluing construction (\cite{Floer4}, \cite{McDS}) gives rise to a one
dimensional family of trajectories starting with $(\beta,\alpha)$ and
ending at some other broken configuration $(\beta',\alpha') \in \M(y,x')
\times \M(x',\emptyset)$. This relation between $(\beta,\alpha)$ and
$(\beta',\alpha')$ will be denoted by
\begin{equation}
  \label{eq:GlueTo}
(\beta,\alpha)\glueto(\beta',\alpha').  
\end{equation}
\begin{remark}
  Recall the gluing construction defines an homeomorphism between a
  neighborhood of the broken configuration $(\alpha,\beta)$ in the
  compactified moduli space $\M(y,\emptyset)$ and
  $\{(\beta,\alpha)\}\times[0,\epsilon)$ for some $\epsilon>0$. In
  particular, this proves that the compactification is a segment and not
  a circle, and hence that relation \eqref{eq:GlueTo} necessarily implies
  that $ (\beta,\alpha)\neq(\beta',\alpha')$.
\end{remark}

This ``move'' from one end of a moduli space to another described above in
$\M(y,\emptyset)$ makes sense for all kinds of configurations, and will be
the main ingredient of all the subsequent constructions. It therefore
deserves a general definition~:
\begin{definition}\label{def:FloerStep}
  A Floer step is an oriented connected component with non empty boundary
  of a $1$-dimensional moduli space.
\end{definition}

\begin{remark}
In particular, the same component defines two steps with opposite
orientations.  
\end{remark}

Depending on the type of moduli space under consideration, there are
several types of Floer steps. Floer loops will be built out of special
steps, called \emph{Floer loop steps}, which are depicted on figure
\ref{fig:FloerLoopSteps} and specified in the following definition~:
\begin{definition}\label{def:FloerLoopStep}
  A Floer loop step is a Floer step in some $\M(y,\emptyset)$ for
  $y\in\CovOrbits_{1}(H)$ or in $\M(\star,\emptyset)$.
\end{definition}

This somewhat abstruse definition is the heart of the construction and
deserves some comments. 

An enlightening point of view is that of  Morse theory.
Consider a function $f$ and a Riemannian metric $g$ on a
  manifold $M$ such that the pair $(f,g)$ is Morse-Smale. Starting with any generic loop in the manifold and
  pushing it down by the flow of $f$ deforms it into the
concatenation of elementary paths, called ``Morse steps'', that
  consist in travelling once, in one or the other direction,
along the unstable manifold of an index $1$ critical point.

It turns out that these steps can be interpreted
from the moduli space point of view~: let $y$ be an index $1$ critical
point and $W^{u}(y)$ its unstable manifold. To a point $p$ in
the unstable manifold is associated a path, namely the
piece of Morse trajectory from $y$ to $p$, and there is a one to one
correspondence between such trajectory pieces and the unstable manifold
(see \cite{BC} for a detailed presentation of this point of view, and a
nice compactification of the unstable manifold derived from it). More
precisely, define an ``interrupted'' Morse trajectory as a solution of
the following modified Morse equation
\begin{equation}\label{eq:InterruptedMorse}
  \frac{d\gamma}{ds} = -\chi(s)\nabla f(\gamma(s))
\end{equation}
where the cutoff function $\chi$ is the same as the one used in
($F_{2}$), i.e. a smooth decreasing function such that $\chi(s)=1$ for
$s\leq-1$ and $\chi(s)=0$ for $s\geq0$.

Using the same notation as in the Floer setting, let
\begin{equation}
  \label{eq:MorseInterruptedModuliSpace}
\MintMorse(y,\emptyset)=\{\gamma:\R\to M, 
\eqref{eq:InterruptedMorse} \text{ and } \lim_{s\to-\infty}\gamma(s)=y\}.
\end{equation}
This space is naturally endowed with an evaluation map (recall the
trajectories are constant for $s\geq0$ so $\gamma(+\infty)=\gamma(0)$),
$$
\begin{array}{ccl}
  \MintMorse(y,\emptyset)&\to&W^{u}(y)\subset M\\
  \gamma&\mapsto&\gamma(+\infty)
\end{array},
$$
which is one to one and provides an identification between
$\MintMorse(y,\emptyset)$ and $W^{u}(y)$. 

\begin{figure}[!ht]
  \centering
  \includegraphics[scale=.4]{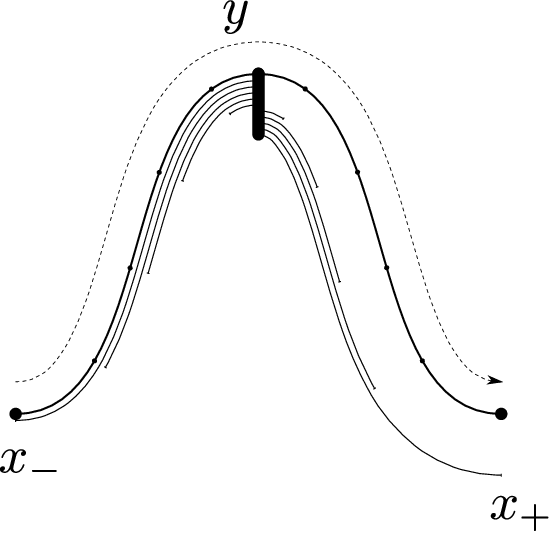}
  \caption{The unstable manifold of an index $1$ Morse critical point seen as a $1$-dimensional moduli space of ``interrupted'' trajectories.}
  \label{fig:MorseStep}
\end{figure}

Moreover, $\MintMorse(y,\emptyset)$ has a natural compactification
$\MMorse(y,\emptyset)$ as a $1$-dimensional segment whose ends are the two
``broken'' configurations $(\gamma_{\pm},\bar{x}_{\pm})$ where
\begin{itemize}
\item 
  $\gamma_{+}$ and $\gamma_{-}$ are the two Morse trajectories rooted at $y$,
\item
  $x_{\pm}$ is the index $0$ critical point such that
  $x_{\pm}=\lim_{s\to+\infty}\gamma_{\pm}$,
\item
  $\bar{x}_{\pm}$ is the constant solution of \eqref{eq:InterruptedMorse}
  at $x_{\pm}$. It is the one and only one augmentation of $x_{\pm}$.
\end{itemize}
The evaluation map extends to this compactification and defines a path
running along $W^{u}(y)$ from $x_{-}$ to $x_{+}$, which is the
``Morse loop step'' associated to $y$.

From this point of view, a Floer loop step through an index $1$ periodic
orbit is the exact counterpart of a Morse loop step through an index 1
critical point.

\begin{remark}
  One noticeable difference between the Morse and Floer settings however,
  is that the Floer moduli space $\M(y,\emptyset)$ need not be
  connected~: each connected component can be interpreted as being one
  ``Floer unstable manifold'' of the orbit $y$, which hence has to be
  considered as as many virtually distinct orbits.
\end{remark}

\begin{remark}
  For orbits $y$ of higher index, the components of the moduli space
  $\M(y,\emptyset)$ can still be regarded as ``Floer unstable manifolds''
  of $y$. However,   there
  is no control a priori on the topology of such a space~: it need not be
  connected, nor need the connected components be balls.
\end{remark}

Similarly, assuming by genericity that $\star$ is not critical for $f$,
the Morse counterpart of the space $\M(\star,\emptyset)$ is the
collection of segments of the (unique) trajectory passing through
$\star$, running from $\star$ down to some arbitrary point $p$ below it
along this trajectory. It is in one to one correspondence with (the
closure of) the piece of trajectory flowing from $\star$  down to the
index $0$ critical point below it.

\bigskip

\begin{figure}[!ht]
  \centering
  \includegraphics[width=11cm]{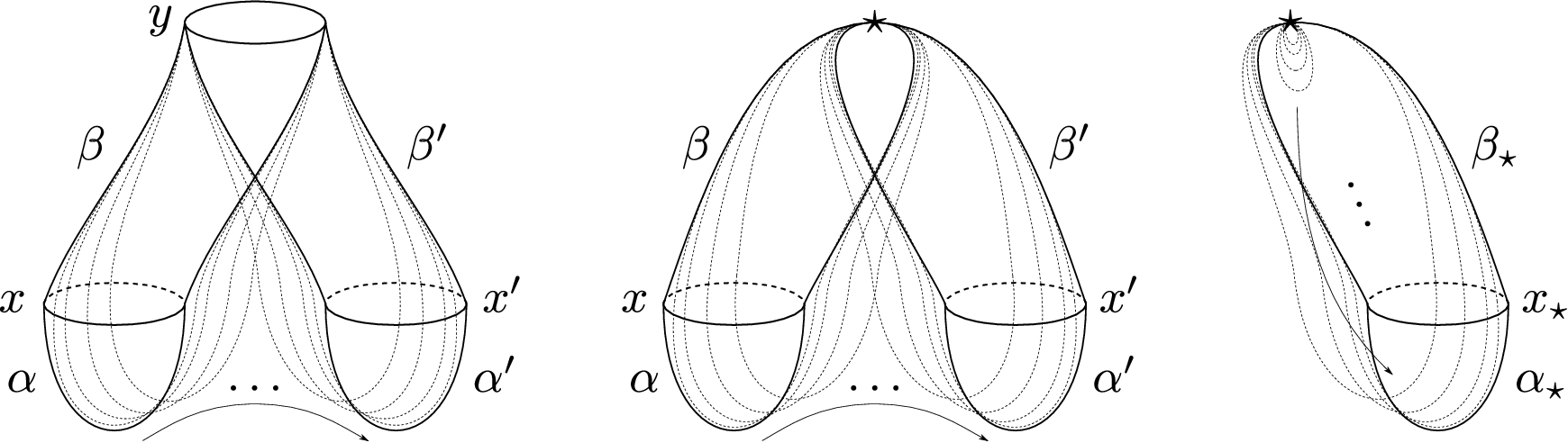}
  \caption{The three kinds of Floer loop steps.}
  \label{fig:FloerLoopSteps}
\end{figure}

\begin{figure}[!ht]
  \centering
  \includegraphics[scale=.5]{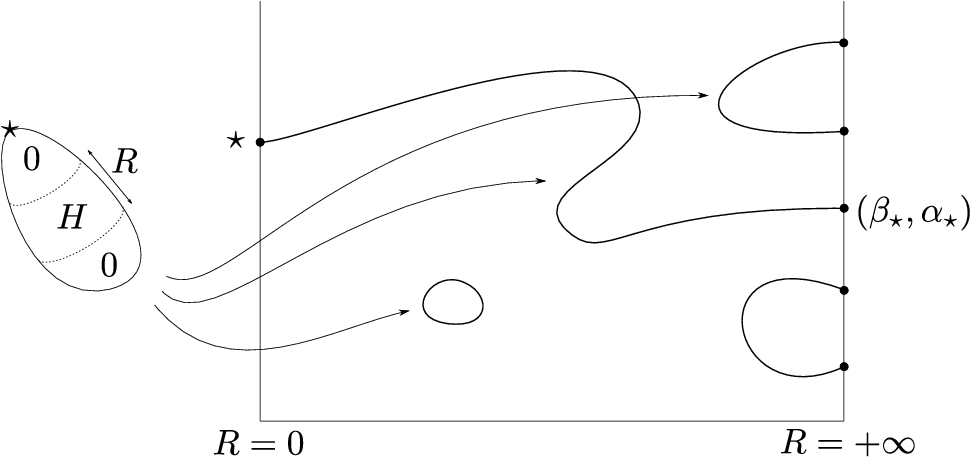}
  \caption{The moduli space $\M(\star,\emptyset)$.}
  \label{fig:BasePoint}
\end{figure}

Definitions \ref{def:FloerStep} and \ref{def:FloerLoopStep} are not very
explicit and a more usable description of a step is obtained by
specifying its ends~:
\begin{proposition}\label{prop:FloerLoopStepsOrbit}
For $y\in\CovOrbits_{1}(H)$, a Floer loop step through $y$ is characterized
by a quadruple $ (\alpha,\beta,\beta',\alpha')$ with
$\alpha\in\M(x,\emptyset)$, $\beta\in\M(y,x)$, $\beta'\in\M(y,x')$, and
$\alpha'\in\M(x',\emptyset) $ for some $x,x'\in\CovOrbits_{0}(H)$ such
that
$$
(\beta,\alpha)\glueto(\beta',\alpha').
$$
\end{proposition}

The situation of Floer loop steps through $\star$ is slightly different,
as there is one special step that does not look like the others.

Recall that the moduli space $\Mint(\star,\emptyset)$ comes with a
projection to the non negative reals
$$
\begin{array}{ccc}
  \Mint(\star,\emptyset)&\xrightarrow{\pi}&[0,+\infty)\\
  (R,u)&\mapsto &R
\end{array}.
$$
This projection is proper, and extends continuously to a map
$\M(\star,\emptyset)\xrightarrow{\pi}[0,+\infty]$ where all the broken
configurations lie above $R=+\infty$. Moreover, the gluing construction
ensures that exactly one component of $\M(\star,\emptyset)$ ends at each 
broken configuration.

Observe now that the same holds over $R=0$~: exactly one component of
$\M(\star,\emptyset)$ ends at the constant map $(u_{\star},0)$. This is a
direct consequence of the regularity of this solution stressed in
proposition \ref{prop:StarIsRegular} (surjectivity of $L$ implies that
$\pi:\M(\star,\emptyset)\to\R$ is a submersion at $(u_{\star},0)$).

As a consequence, $\M(\star,\emptyset)$ has exactly one connected
component that relates $\{\star\}$ to a broken configuration, and all the
other components either have no boundary or relate two broken
configurations~:
\begin{equation}
  \label{eq:AlphaStar}
\exists! x_{\star}\in\CovOrbits_{0}(H),
\exists!(\beta_{\star},\alpha_{\star})\in\M(\star,x_{\star})\times\M(x_{\star},\emptyset),
 \quad 
(\beta_{\star},\alpha_{\star})\glueto\star
\end{equation}


\begin{proposition}\label{prop:FloerLoopStepsStar}
There are exactly one orbit $x_{\star}\in\CovOrbits_{0}(H)$
and one pair $(\beta_{\star},\alpha_{\star}) \in
\M(\star,x_{\star})\times\M(x_{\star},\emptyset)$ such that a Floer loop
step through $\star$ is 
\begin{itemize}
\item 
  either the special step $\star \glueto (\beta_{\star},\alpha_{\star})$ 
\item
  or characterized by a quadruple $ (\alpha,\beta,\beta',\alpha')$ with
  $\alpha\in\M(x,\emptyset)$, $\beta\in\M(\star,x)$,
  $\beta'\in\M(\star,x')$, and $\alpha'\in\M(x',\emptyset) $ for some
  $x,x'\in\CovOrbits_{0}(H)$ such that
  $$
  (\beta,\alpha)\glueto(\beta',\alpha').
  $$
\end{itemize}
\end{proposition}

\begin{remark}
  Considering loop steps entering the second case in the above statement
  might seem unnatural since, as already mentioned, in the Morse setting,
  only the special step $\star \glueto (\beta_{\star},\alpha_{\star})$
  does exist. In the Floer context however, as well as in the stable
  Morse setting where examples are much easier to produce (see section
  \ref{sec:StableMorse} and figure \ref{fig:StableMorseStep}), such steps
  might exist, and have to be taken into account.
\end{remark}

\begin{remark}
  Notice there are only finitely many Floer loop steps~:
  there are finitely many periodic orbits, and because of the
  monotonicity assumption, finitely many lifts of each can have index $0$
  or $1$, and finally, each $0$-dimensional moduli space is compact and
  hence finite. 
\end{remark}

\bigskip

Notice finally that Floer loop steps are oriented  and hence have
  a start and an end~: 
\begin{definition}
 With the notations of propositions \ref{prop:FloerLoopStepsOrbit} and
 \ref{prop:FloerLoopStepsStar}, a step $(\alpha,\beta,\beta',\alpha')$ is
 said to start at $\alpha$ and end at $\alpha'$. 

 Similarly, if one of the pairs $(\beta,\alpha)$ or $(\beta',\alpha')$
 is replaced by $\star$, the corresponding end is said to be $\star$ itself.
 
 Two loop steps are said to be consecutive if the end of the first is the
 start of the second.  
\end{definition}

\begin{figure}[!ht]
  \centering
  \includegraphics[scale=.6]{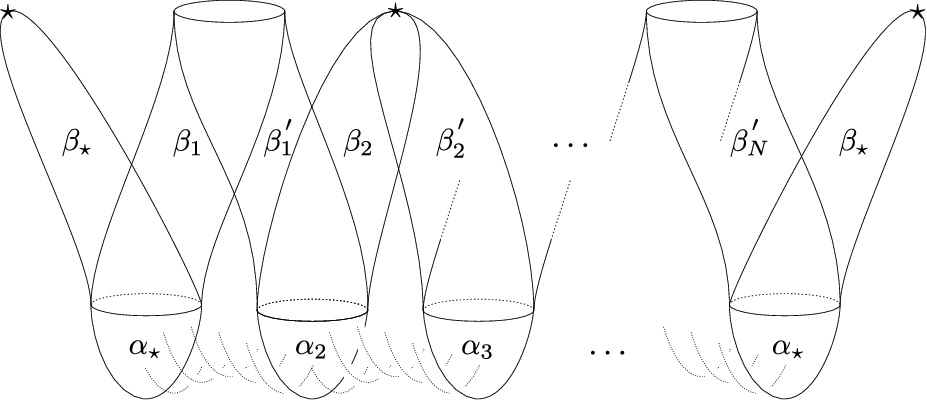}
  \caption{A Floer based loop.}
  \label{fig:FloerPath}
\end{figure}

\begin{definition}\label{def:FloerLoops}
A Floer based loop is a sequence of consecutive Floer loop steps starting
and ending at $\star$.
\end{definition}

In other words, a Floer based loop is a sequence
\begin{gather*}
(\star,\beta_{\star},\alpha_{\star}),
(\alpha_{\star},\beta_{1},\beta'_{1},\alpha_{2}),
(\alpha_{2},\beta_{2},\beta'_{2},\alpha_{3}),\dots,
(\alpha_{N},\beta_{N},\beta'_{N},\alpha_{\star}),
(\alpha_{\star},\beta_{\star},\star)
\intertext{such that (letting $\alpha_{1}=\alpha_{N+1}=\alpha_{\star}$)~:}
\forall i\in\{1,\dots, N+1\},\quad
(\beta_{i},\alpha'_{i})\glueto(\beta'_{i},\alpha_{i+1}), 
\end{gather*}

Let $\GenBrut{H}$ be the set of all Floer based loops. Notice it depends
on all the auxiliary data $(H,\star,J,\cutoff)$ but the
dependency on $J$ and $\cutoff$ is kept  implicit to reduce the
notation. It carries an obvious concatenation rule that turns it into a
semi-group.

It also carries obvious cancellation rules. More explicitly, if
$\sigma=(\alpha,\beta,\beta',\alpha')$ is a Floer loop step, define its
inverse $\sigma^{-1}$ to be the same step with the opposite orientation~:
$$
\sigma^{-1}=(\alpha',\beta',\beta,\alpha).
$$
Denote by $\sim$ the associated cancellation rules in $\GenBrut{H}$~:
$$
\sigma_{1}\dots \sigma_{i}\sigma^{-1}_{i}\dots\sigma_{N} \sim 
\sigma_{1}\dots \sigma_{i-1}\sigma_{i+1}\dots\sigma_{N}.
$$
The concatenation then endows the quotient space
\begin{equation}
  \label{eq:GenGp}
  \Gen{H}=\faktor{\GenBrut{H}}{\sim}.
\end{equation}
with a group structure.

\bigskip

A Floer loop step being a one parameter family of tubes, evaluation at
the $+\infty$ end of the tube defines a path in $M$ (an arbitrary
parameterization can be chosen for each step, since we are only
interested in the resulting homotopy class), and induces a map
\begin{equation}
  \label{eq:evalBrut}
  \GenBrut{H}\xrightarrow{\eval}
  \pi_{1}(M,\star).
\end{equation}
This map is compatible with both the concatenation and the
cancellation rules and hence induces a group homomorphism
\begin{equation}
  \label{eq:eval}
  \Gen{H}\xrightarrow{\eval}\pi_{1}(M,\star).
\end{equation}

All the objects involved in theorem \ref{thm:RoughSurjectivity} are now
defined  and we recall its statement~:
\begin{theorem}\label{thm:EvalOntoPi1}
  With the above notations, the evaluation map induces a surjective homomorphism
  \begin{equation}
    \label{eq:onto}
    \xymatrix{\Gen{H} \ar@{>>}[r]^{\eval}&\pi_{1}(M,\star)}.
  \end{equation}
\end{theorem}

The description of the relations still requires the introduction of
further
  technical ingredients, and we postpone it to section
\ref{sec:Relations} to focus in the next section on the application to
the count, with multiplicity, of Hamiltonian periodic orbits, since it
only requires the surjectivity.

\subsection{Application}

\begin{definition}\label{def:Multiplicity}
Define the multiplicity of a Hamiltonian orbit $y\in\CovOrbits_{1}(H)$ as
the number of steps through it, i.e.
$$
 \nu_{J}(y) = 
 \frac{1}{2}\sum_{x\in\CovOrbits_{0}(H)}
   \Cardabs\M(y,x)\cdot\Cardabs\M(x,\emptyset).
$$ 
Define the multiplicity of the point $\star$ as the number
$$
  \nu_{J}(\star) = \frac{1}{2}
  \Big(\sum_{x\in\CovOrbits_{0}(H)}
  \Cardabs\M(\star,x)\cdot\Cardabs\M(x,\emptyset)\Big)
  -\frac{1}{2}.
$$ 
\end{definition}
Notice the counting here is not algebraic but geometric~: it is not hard
to see that the algebraic count would always be $0$.

\begin{remark}
Although they may seem to be $\frac{1}{2}\N$ valued, these numbers are in
fact integer valued~: as already observed, the gluing construction groups
the broken trajectories $(\beta,\alpha)$ from some
$y\in\CovOrbits_{1}(H)$ to $\emptyset$ in pairs, so there is an even
number of such, and the same holds for broken trajectories from $\star$
to $\emptyset$ but for $(\beta_{\star},\alpha_{\star})$, which proves
there is an odd number of such configurations.  
\end{remark}

\begin{remark}\label{rk:nustar}
  For $x\in\CovOrbits_{0}(H)$, letting $\nu_{J,\star}(x)=\frac{1}{2} \cdot
  \Cardabs\M(\star,x) \cdot \Cardabs\M(x,\emptyset)$,   we have
  $\nu_{J}(\star)+1=\sum_{x\in\CovOrbits_{0}(H)}\nu_{J,\star}(x)$, so that
  $\nu_{J}(\star)$ can also be expressed as a sum of
  ($\frac{1}{2}\N$-valued) ``multiplicities'' of index $0$ periodic orbits. 
\end{remark}

\begin{remark}\label{rk:MorseMultiplicity}
  Recall from \eqref{eq:MorseInterruptedModuliSpace} that all the
  involved moduli spaces, and hence the notion of multiplicity itself,
  make sense in the Morse setting. However, the Morse situation is much
  more constrained, and we know there are exactly two trajectories rooted
  at each index $1$ critical point, and exactly one through $\star$~:
  this implies that in the Morse setting, the multiplicity is always $1$
  for index $1$ critical points and $0$ for $\star$.
\end{remark}

The following statement is a reformulation of theorem
\ref{thm:RoughEnoughPoints} and is a direct corollary of our
construction. It will be proven in section \ref{sec:ProofEnoughSteps}.
\begin{theorem}\label{thm:EnoughSteps}
Let $\rho(\pi_{1}(M))$ be the minimal number of elements in a generating
family of $\pi_{1}(M)$. Then
\begin{equation}
  \label{eq:MainEstimate}
  \nu_{J}(\star)\ +
  \sum_{y\in\CovOrbits_{1}(H)} \nu_{J}(y)\geq \rho(\pi_{1}(M)).  
\end{equation}
In other words, counted with multiplicities, $
\{\star\}\cup\CovOrbits_{1}(H)$ contains sufficiently many elements to
generate $\pi_{1}(M)$.
\end{theorem}

\begin{remark}
  According to remark \ref{rk:MorseMultiplicity}, the left hand side in
  \eqref{eq:MainEstimate} in the Morse setting is exactly the number of
  index $1$ critical points, so that in this setting the inequality
  \eqref{eq:MainEstimate} is nothing but the usual lower estimate of the
  number of index $1$ critical points of a Morse function
  by the minimal number of generators of the fundamental group $\pi_{1}(M)$.%
\end{remark}


\begin{remark}
  The term $\nu_{J}(\star)$ in \eqref{eq:MainEstimate} may be unexpected,
  since it automatically vanishes in the Morse setting. It is a very
  natural question to ask how essential it is and if it can be
  controlled. 

  The  theorem \ref{thm:RoughAtLeastOneOrbit}, stated in a more precise
  form below as theorem \ref{thm:AtLeastOneOrbit} and proven in section
  \ref{sec:ProofAtLeastOneOrbit}, ensures that when $\pi_{1}(M)\neq\{1\}$,
  the contribution of the index $1$ orbits is at least $1$, since it
  provides at least one such orbit with non vanishing multiplicity.

  This lower bound on the number of index $1$
    orbits may seem rather small, but no better result seems to be known
    without further assumption on the fundamental group yet. Moreover,
    the proof itself is very  geometric and might be of independent
  interest~: it is a variation, in the usual context of PSS moduli
  spaces, on the main guiding principle of this paper of using
  $1$-dimensional moduli spaces to catch extra information. We stress
  however that it is not an application of the construction of the
  fundamental group, but an illustration that the multiplicities cannot
  be arbitrary.
\end{remark}

\begin{theorem}\label{thm:AtLeastOneOrbit}
  Suppose $\pi_{1}(M)\neq\{1\}$. Let $H$ be a non degenerate Hamiltonian
  function and $J$ a generic choice of a time dependent almost complex
  structure $J$ compatible with $\omega$.

  Then $H$ has at least one contractible $1$-periodic orbit with
  Conley-Zehnder index $1-n$ and with non vanishing multiplicity with
  respect to $J$.
\end{theorem}

\subsection{More notations and tools.} 

\subsubsection{Mixed moduli spaces}
In addition to the already introduced moduli spaces we will need hybrid
Morse-Floer moduli spaces, depicted in figure \ref{fig:MixedModuliSpaces}
and defined below. 
\begin{figure}[!ht]
  \centering
  \includegraphics[scale=.6]{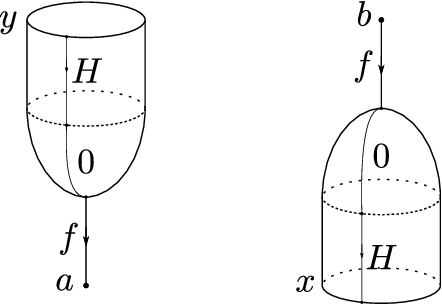}
  \caption{Hybrid moduli spaces.}
  \label{fig:MixedModuliSpaces}
\end{figure}

Let $f$ be a Morse function and $g$ a Riemannian metric on $M$.
By convention, the Morse flow associated to $(f,g)$ is the flow
  of the negative gradient $-\nabla f$ of $f$ with respect to $g$. Let
$\Crit_{k}(f)$ be the set of index $k$ critical points, and suppose
$\Crit_{0}(f)=\{\star\}$. For $y\in\CovOrbits_{1}(H)$ and $a\in\Crit(f)$,
we let 
$$
\Mint(y,a) = \{u\in\M(y,\emptyset), u(+\infty)\in W^{s}(a)\},
$$
where $W^{s}(a)$ is the stable manifold of $a$. 

Similarly, for $b\in\Crit_{1}(f)$ and $x\in\CovOrbits_{0}(H)$, we let 
$$
\Mint(b,x) = \{u\in\M(\emptyset,x), u(-\infty)\in W^{u}(b)\},
$$ 
where $W^{u}(b)$ is the unstable manifold of $b$.

The couple $(f,g)$ is supposed to be
chosen generically, so that all these spaces are cut out transversely. In
particular, they have the expected dimensions~:
%
%
%
%
$$
\dim\Mint(y,a) = |y|-|a|
\qquad
\dim\Mint(b,x) = |b|-|x|,
$$
(where the Morse index is also denoted by $|\cdot|$). Moreover, these
spaces are compact up to bubbling of spheres and breaking, either at an
intermediate Hamiltonian orbit or at an intermediate Morse critical point
(see \cite{PSS}), and the compactifications are denoted by $\M(y,a)$ and
$\M(b,x)$. When they are $0$ or $1$-dimensional, no bubbling can occur on
such moduli spaces, and they consist of a finite set of points when they
are $0$-dimensional, and a finite set of circles and segments whose
boundary consists in broken configurations when they are $1$-dimensional.
Finally, recall there  is a gluing construction proving every broken
configuration does indeed appear on the boundary of a bigger moduli space.

\subsubsection{Crocodile walk.}\label{sec:CrocodileWalk}
We now introduce the main technical tool.

Consider a Hamiltonian orbit $z$ of index $2$. Let $B(z)$ be the space of
twice broken trajectories from $z$ to $\emptyset$~:
$$
B(z)=\bigcup_{\substack{|y|=1\\|x|=0}}\M(z,y)\times\M(y,x)\times\M(x,\emptyset).
$$
For each such trajectory $(\gamma,\beta,\alpha)$ in some
  $\M(z,y)\times\M(y,x)\times\M(x,\emptyset)$, the gluing construction
can take place either at the upper breaking $y$ or at the 
lower one $x$. Gluing at the upper breaking defines an
involution
$$
\sharp^{\bullet}\begin{array}{cccl}
  B(z) & \to & B(z) \\
  (\gamma,\beta,\alpha)& \mapsto & (\gamma',\beta',\alpha)
\end{array},
$$
where $(\gamma',\beta')$ is such that
$(\gamma,\beta)\glueto(\gamma',\beta')$. Similarly, gluing at the lower
breaking, defines another involution
$$
\sharp_{\bullet}\begin{array}{cccl}
  B(z) & \to & B(z) \\
  (\gamma,\beta,\alpha)& \mapsto & (\gamma,\beta',\alpha')
\end{array}.
$$

According to definition \ref{def:FloerStep}, upper and lower gluings are
both Floer steps, and lower gluings are Floer loop steps.

Iteration of alternately upper and lower gluings then naturally appears as
a walk on the space of twice broken trajectories. Moreover, since the
intermediate Floer trajectory form a zigzag pattern (see figure
\ref{fig:CrocoWalk}), we use the following vocabulary~:
\begin{definition}
Iteration of alternately upper and lower gluings
$\sharp_{\bullet}\circ\sharp^{\bullet}\circ\sharp_{\bullet}\circ\sharp^{\bullet}\circ\dots$ will be abbreviated as
running a \emph{``crocodile walk''} on the set $B(z)$ of twice broken
trajectories from $z$ to $\emptyset$.   
\end{definition}

\begin{remark}
  Given a twice broken configuration, the crocodile walk can be started
  with an upper or a lower gluing~: because $\sharp^{\bullet}$ and
  $\sharp_{\bullet}$ are involutions, this only affects the walking
  direction along the orbit, but not the underlying non-oriented orbit.
  We consider orbits as oriented however, so through one configuration go
  exactly two orbits of the crocodile walk, which differ only by the orientation. 
\end{remark}

  \begin{remark}
    A more geometric interpretation of the crocodile walk can be given by
    considering the boundary components of the $2$-dimensional  moduli
    space $\M(z,\emptyset)$ (see figure \ref{fig:CornersWalk}). The set
    $\M(z,\emptyset)\setminus\Mint(z,\emptyset)$ consists in bubbling
    configurations, which are $2$ codimensional, and ``boundary
    components'' which are $1$ codimensional and consist in broken
    configurations. The latter components are circles that are either
    smooth (when they are the product of two smaller moduli spaces
    without boundary) or have ``corners'' at twice broken configurations.
    The crocodile walk consists in moving along such an ``angular''
    boundary component from one corner to the next.
  \end{remark}

\begin{figure}[!ht]
  \centering
  \includegraphics[scale=.6]{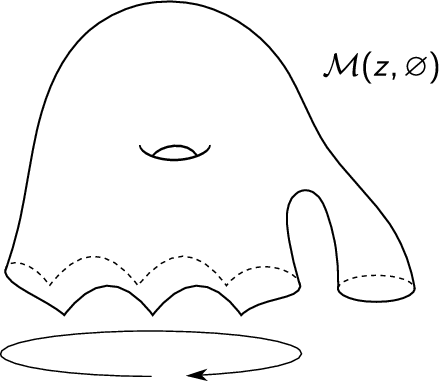}
  \caption{The crocodile walk as a way to explore ``angular'' boundary 
    components of $2$-dimensional moduli spaces
    (here in $\M(z,\emptyset)$ for
    $z\in\CovOrbits_{2}(H)$).}
  \label{fig:CornersWalk}
\end{figure}

\begin{remark}
  Crocodile walks can in fact be defined  on
  any kind of $0$-dimensional moduli space of twice broken
  configurations, like the space of twice broken Floer trajectories
  between orbits of relative index $3$ for instance, or hybrid moduli
  spaces mixing Floer and Morse trajectories as in the next paragraph.
\end{remark}

The crocodile walk is the iteration of a one to one map
($\sharp_{\bullet}\circ\sharp^{\bullet}$) on a finite set, so the
orbits all have to be cyclic. 

Moreover, if a configuration is reached after an upper (resp. lower)
gluing, it has to be left with a lower (resp. upper) one. As a
consequence, being cyclic, an orbit has to contain the
same number of upper and lower gluings. In particular, it counts an even
number of steps.

\medskip

To an orbit of the crocodile walk is not only associated a sequence of
twice broken trajectories, but also an abstract polyhedron representing
the way the trajectories in the different moduli spaces fit together.

\begin{figure}[!ht]
  \centering
  \includegraphics[scale=.4]{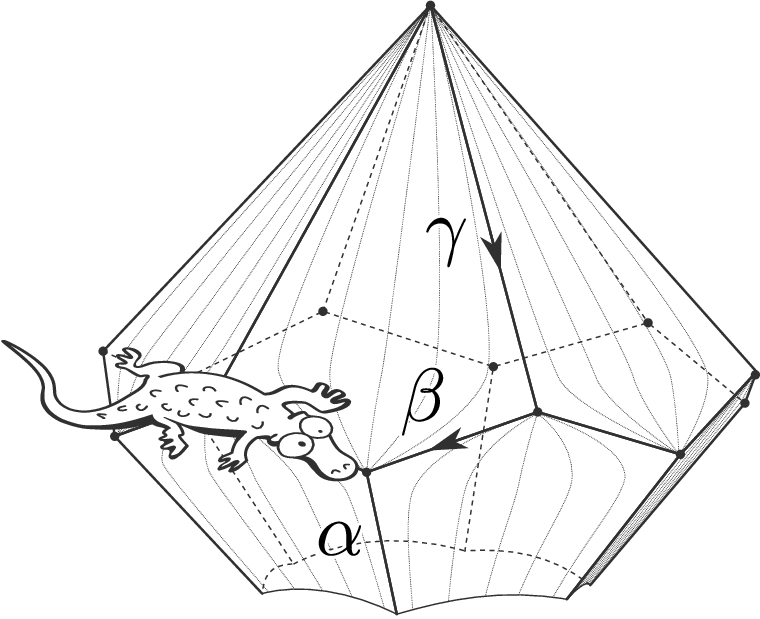}
  \caption{An orbit of the crocodile walk on the space of twice broken trajectories
    from $z\in\CovOrbits_{2}(H)$ to $\emptyset$.}
  \label{fig:CrocoWalk}
\end{figure}


An orbit $W$ of the crocodile walk is a sequence
$$
(
(\gamma_{1},\beta_{1},\alpha_{1}),
(\gamma_{2},\beta'_{1},\alpha_{1}),
(\gamma_{2},\beta_{2},\alpha_{2}), \dots,
(\gamma_{N},\beta_{N},\alpha_{N})
)
$$
such that $(\gamma_{1},\beta_{1},\alpha_{1}) =
(\gamma_{N},\beta_{N},\alpha_{N})$ and
\begin{gather}
  (\gamma_{k},\beta_{k})\glueto(\gamma_{k+1},\beta'_{k})
  \quad\text{ and }\quad
  (\beta'_{k},\alpha_{k})\glueto(\beta_{k+1},\alpha_{k+1}).
\end{gather}

\begin{lemma}
  Let $W$ be an orbit of the crocodile walk like above. There exists an
  abstract disc $\Delta(W)$ endowed with a continuous map
  $\Delta(W)\xrightarrow{\eval}M$ whose restriction to the boundary is
  the concatenation of evaluation of the Floer steps
  $(\beta'_{1},\alpha_{1})\glueto(\beta_{2},\alpha_{2}), \dots,
  (\beta'_{N-1},\alpha_{N-1})\glueto(\beta_{N},\alpha_{N}) $
\end{lemma}

\begin{proof}
  Let $\M^{\bullet}_{k}$ (resp. $\M_{\bullet k}$) be an abstract copy of
  the component of the moduli space relating $(\gamma_{k},\beta_{k})$ to
  $(\gamma_{k+1},\beta'_{k})$(resp. $(\beta'_{k},\alpha_{k})$ to
  $(\beta_{k+1},\alpha_{k+1})$). Let $\Sigma\M^{\bullet}_{k}$ be it's
  suspension~: it is the suspension of a segment and hence can be
  identified with the standard diamond. 

  Recall that before compactification, the evaluation along the real line
  $\R\times\{0\}\subset\R\times\S^{1}$ defines a map
  $$
  \overset\circ {\M^{\bullet}_{k}}\times\R\xrightarrow{\eval}M.
  $$
  Since the action is strictly decreasing along the Floer trajectories,
  it can be used to define a parameterization of the trajectories, and to
  define a continuous map
  $$
  \overset\circ {\M^{\bullet}_{k}}\times[-1,1]\xrightarrow{\eval}
  M
  $$
  that extends continuously to the compactification, and descends to the
  suspension
  $$
  \Sigma\M^{\bullet}_{k}\xrightarrow{\eval}M.
  $$

  \begin{figure}[!ht]
    \includegraphics[scale=.4]{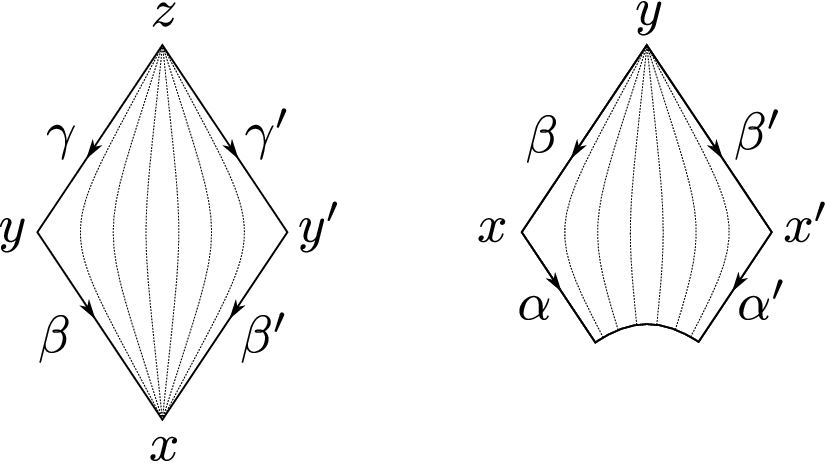}
    \centering
    \caption{Steps suspensions.}
    \label{fig:Suspensions}
  \end{figure}

  We think of $\Sigma\M^{\bullet}_{k}$ as a diamond (see figure
  \ref{fig:Suspensions}), and on the four sides, the evaluation map is
  the action-normalized evaluation along the broken trajectories
  $(\gamma_{k},\beta_{k})$ on the left and $(\gamma_{k+1},\beta'_{k})$ on
  the right. 

  \medskip

  A similar construction can also be achieved for the $\M_{\bullet k}$
  spaces. The lower end of the trajectories is not constrained however,
  and the suspension should be replaced by the half suspension
  $\Sigma'\M_{\bullet k}=\faktor{\M_{\bullet k}\times[-1,1]} {
    \M_{\bullet k} \times\{1\}}$. We think of this as a truncated
  diamond, or a pentagon (see figure \ref{fig:Suspensions}). It is
  endowed with an evaluation map whose restriction
  \begin{itemize}
  \item 
    to the upper left side (i.e.
    $[0,1]\times\{(\beta'_{k},\alpha_{k})\}$) is $\beta'_{k}$
  \item
    to the lower left side (i.e.
    $[-1,0]\times\{(\beta'_{k},\alpha_{k})\}$) is $\alpha_{k}$
  \item
    to the upper right side (i.e.
    $[0,1]\times\{(\beta_{k+1},\alpha_{k+1})\}$)is $\beta_{k+1}$
  \item
    to the lower right side (i.e.
    $[-1,0]\times\{(\beta_{k+1},\alpha_{k+1})\}$)is $\alpha_{k+ 1}$
  \item
    to the bottom side (i.e. $\{-1\}\times\M_{\bullet k}$) is the
    evaluation at the center of the augmentations $\eval(u)= u(+\infty)$
  \end{itemize}

  \begin{figure}[!ht]
    \centering
    \includegraphics[scale=.4]{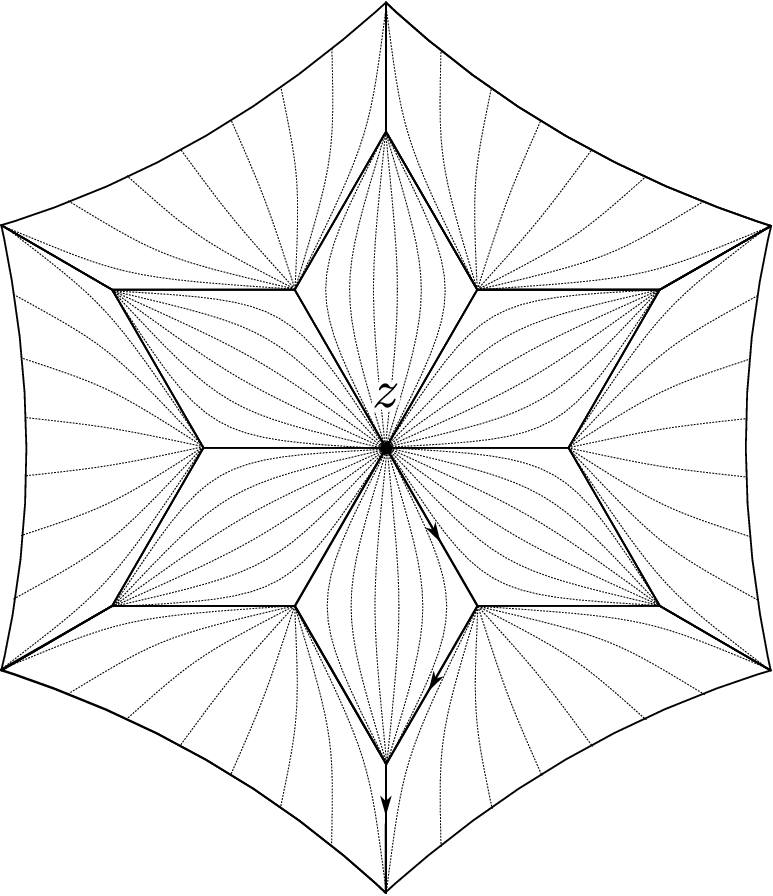}
    \caption{The disc $\Delta(W)$.}
    \label{fig:CrocoDisc}
  \end{figure}

  We identify all these diamonds and pentagons along their shared sides
  in the order of the gluings appearing in the orbit $W$ (see figure
  \ref{fig:CrocoDisc}). Formally, we let
  \begin{equation}
    \label{eq:CrocoDisc}
    \Delta(W)=\faktor{\Big(
      \bigsqcup_{k=1}^{N}\Sigma\M^{\bullet}_{k}\sqcup\Sigma'\M_{\bullet k}
      \Big)}{\sim}
  \end{equation}
  where $\sim$ is the identification, for each $k$ of 
  \begin{itemize}
  \item 
    the upper right side of $\Sigma\M^{\bullet}_{k}$ with the upper left
    side of $\Sigma\M^{\bullet}_{k+1}$,
  \item 
    the lower right side of $\Sigma\M^{\bullet}_{k}$ with the upper left
    side of $\Sigma\M_{\bullet k+1}$.
  \item 
    the lower left side of $\Sigma\M^{\bullet}_{k+1}$ with the upper
    right side of $\Sigma\M_{\bullet k+1}$.
  \end{itemize}
  
  The resulting $2$-dimensional polyhedron $\Delta(W)$ is a disc.
  Moreover, since it is compatible with all the identifications, the
  evaluation map descends to $\Delta(W)$ and defines a continuous map 
  \begin{equation}
    \label{eq:EvaluationOnCrocoDisc}
    \Delta(W)\xrightarrow{\eval} M.
  \end{equation}
  and has the desired behaviour on the boundary.
\end{proof}

\begin{remark}
Reversing the orientation of $W$ reverses the orientation of the
associated disc.
\end{remark}

\begin{remark}
  Regarding the crocodile walk orbit $W$ as a boundary component of a
  $2$-dimensional moduli space, the disc $\Delta(W)$ is essentially the
  same as the half suspension of this boundary component. 

  This geometric point of view does not avoid the above description
  however, since the structure of the disc and in particular the behavior
  of the evaluation on its boundary is crucial to our construction. 
\end{remark}

\subsubsection{Hybrid walks}\label{sec:HybridWalks}
As already observed, the ``crocodile walk''  can in fact be run on many
kinds of moduli spaces, in particular on a hybrid moduli space mixing
Morse trajectories rooted at an index $1$ critical point of our Morse
function $f$ and Floer tubes.

Let $b\in\Crit_{1}(f)$, let
$\{\gamma_{-},\gamma_{+}\}=\M(b,\star)$ be the two Morse trajectories
rooted at $b$ (recall $\Crit_{0}(f)=\{\star\}$). Let 
$$
B(b)=\bigcup_{y\in\CovOrbits_{1}(H)\cup\{\star\}}\M(b,y)\times\partial\M(y,\emptyset).
$$
This space plays the role of twice broken trajectories, but as already
observed, the space $\partial\M(\star,\emptyset)$ has one (and only one)
point which is not a breaking~: $B(b)$ splits as the union $B(b)=B'(b)\cup
B_{\star}(b)$ of the set of twice broken trajectories
$$
B'(b)=\bigcup_{\substack{
    y\in\CovOrbits_{1}(H)\cup\{\star\}\\
    x\in\CovOrbits_{0}(H) }}\M(b,y)\times\M(y,x)\times\M(x,\emptyset),
$$
and the two special isolated configurations that are not broken twice~:
$$
B_{\star}(b)=\M(b,\star)\times\{\star\}=\{(\gamma_{-},\star),(\gamma_{+},\star)\},
$$
where $\star$ is seen as the constant sphere in
$\M(\star,\emptyset)$.

Upper and lower gluings can be performed on $B'(b)$, but have to be
replaced by the relevant Floer steps on $B_{\star}(b)$, and we let
\begin{equation}  \label{eq:specialSteps}
  \begin{aligned}
  \sharp^{\bullet}(\gamma_{\pm},\star) &= (\gamma_{\mp},\star)\\ 
  \sharp_{\bullet}(\gamma_{\pm},\star) &= (\gamma_{\pm},\beta_{\star},\alpha_{\star})
\end{aligned}
\end{equation}

If the latter was already discussed, observe the former is rather a Morse
step. To see it as a Floer step, consider the moduli space of solutions
$(u,R)$ of $(F_{4,R})$ such that $u(-\infty)\in W^{u}(b)$, but restrict
attention to the boundary component given by $R=0$~: the configurations
$(\gamma_{\pm},\star)$, regarded as such configurations that underwent
a Morse breaking, are related by the moduli space obtained by gluing at the
Morse breaking and preserving the $R=0$ condition.

Defined in this way, the maps $\sharp^{\bullet}$ and $\sharp_{\bullet}$
form two involutions on $B(b)$ again, and iterated composition of
alternately $\sharp^{\bullet}$ and $\sharp_{\bullet}$ defines a walk,
still called a crocodile walk, whose orbits are all cyclic. 

\begin{remark}\label{rk:CrocodileLowerStepsAreLoopSteps}
  Notice for later use that the steps used in the definition of  
  $\sharp_{\bullet}$ are all Floer loop steps.  
\end{remark}

\begin{figure}[!ht]
  \centering
  \includegraphics[scale=.4]{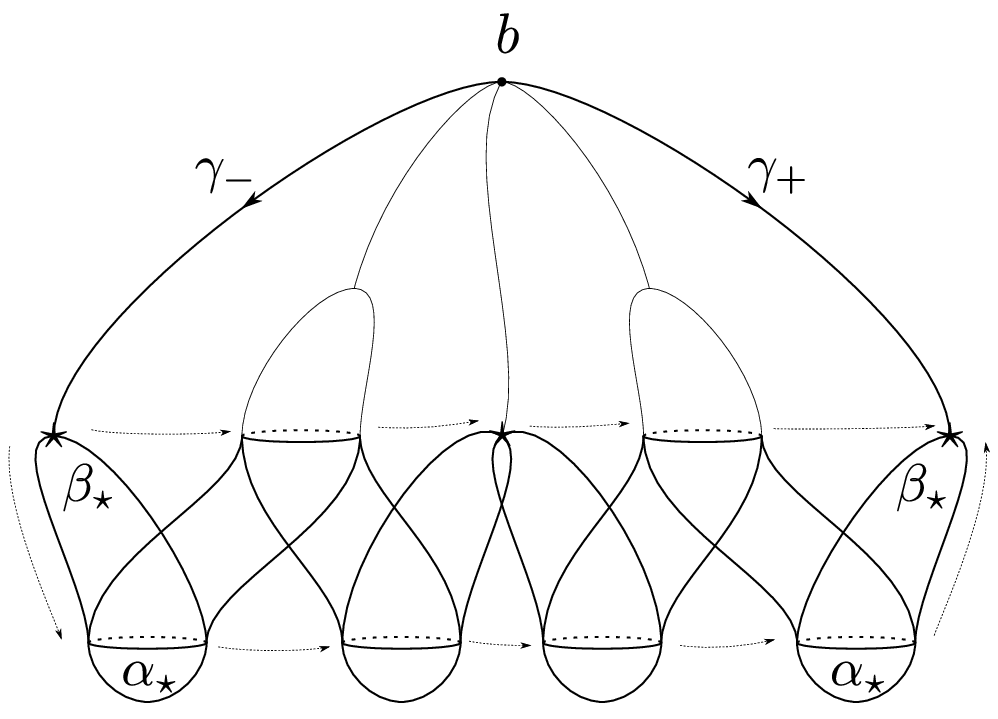}
  \caption{The orbit $W_{\star}(b)$ of the crocodile walk 
    associated to an index $1$ Morse critical point $b$.}
  \label{fig:MorseToFloer}
\end{figure}

\begin{definition}\label{def:CrocoStop}
  The orbit of $(\gamma_{-},\star)$ starting with a lower gluing will be
  denoted by $W_{\star}(b)$.
\end{definition}
 It is a cyclic sequence of the following form
  $$
  (\gamma^{-},\star),
  (\gamma^{-},\beta_{\star},\alpha_{\star}),
  (\gamma_{1},\beta_{1},\alpha_{1}), \dots
  (\gamma_{N},\beta_{N},\alpha_{N}),
  (\gamma^{+},\beta_{\star},\alpha_{\star}),%
  (\gamma^{+},\star),
  $$
where
\begin{itemize}
\item 
  $N$ is even (the orbits being cyclic, they have to count the same number of
  upper and lower steps, and hence an even number of elements), 
\item
  for $1\leq i\leq N$
  \begin{itemize}
  \item 
    $\gamma_{i}\in\M(b,y)$ for some $y\in\{\star\}\cup\CovOrbits_{1}(H)$,
  \item
    $\beta_{i}\in\M(y,x)$ for some $x\in\CovOrbits_{0}(H)$,
  \item
    $\alpha_{i}\in\M(x,\emptyset)$,
  \end{itemize}
\item
  for all $i$ with $0\leq i< N/2$~:
  \begin{gather*}
    (\gamma_{2i},\beta_{2i})\glueto(\gamma_{2i+1},\beta_{2i+1})
    \qquad\text{ and }\qquad
    \alpha_{2i}=\alpha_{2i+1}\\
    \gamma_{2i+1}=\gamma_{2i+2}  
    \quad\text{ and }\quad
    (\beta_{2i+1},\alpha_{2i+1})\glueto(\beta_{2i+2},\alpha_{2i+2}).
  \end{gather*}
  (with the convention
  $(\gamma_{0},\beta_{0})=(\gamma_{-},\beta_{\star})$ and 
  $(\gamma_{N+1},\beta_{N+1})=(\gamma_{+},\beta_{\star})$).
\end{itemize}

In particular (recall remark \ref{rk:CrocodileLowerStepsAreLoopSteps}), 
the sequence of lower steps 
\begin{equation}\label{eq:LowerCrocoSteps}
(\star,\beta_{\star},\alpha_{\star}),
\dots
(\alpha_{2i-1},\beta_{2i-1},\beta_{2i},\alpha_{2i}),
\dots
(\alpha_{\star},\beta_{\star},\star)
\end{equation}
form a Floer loop.

\begin{figure}[!ht]
  \centering
  \includegraphics[scale=.4]{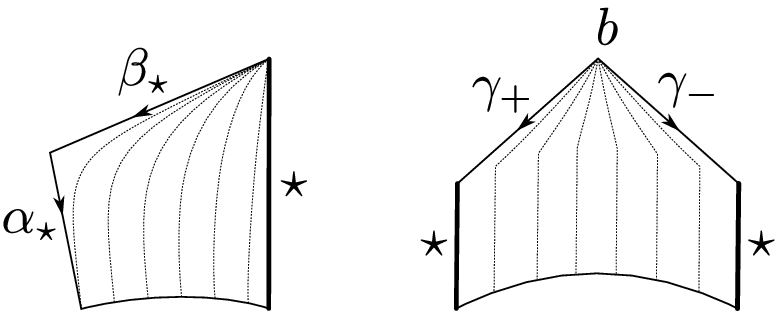}
  \caption{Half suspensions of the steps $(\beta_{\star},\alpha_{\star})\glueto \star$ and $(\gamma_{+},\star)\glueto(\gamma_{-},\star)$ }
  \label{fig:SpecialSuspensions}
\end{figure}

\begin{figure}[!ht]
  \centering
  \includegraphics[scale=.4]{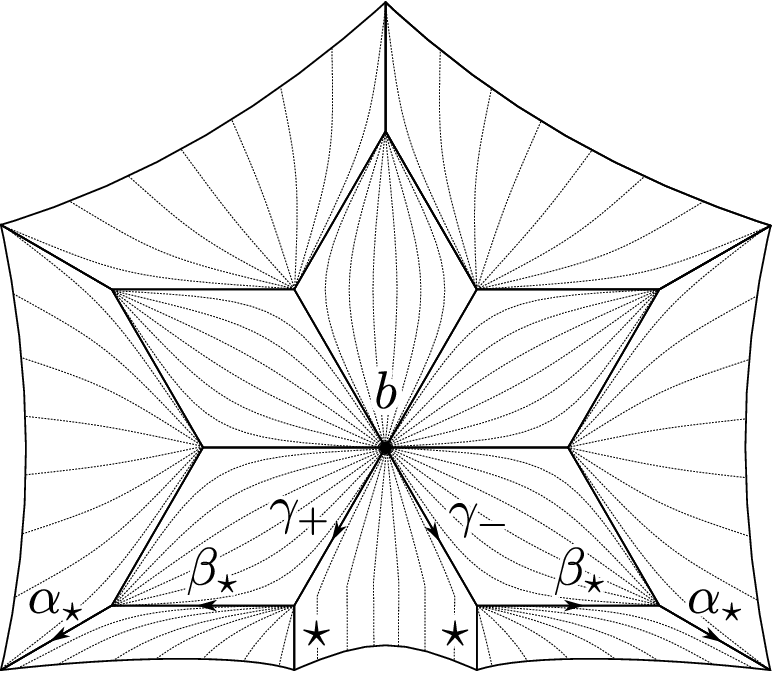}
  \caption{The disc $\Delta(W_{\star}(b))$}
  \label{fig:SpecialCrocoDisc}
\end{figure}

The construction of the polyhedron $\Delta(W_{\star}(b))$ still makes
sense for this special orbit~: exactly two new kinds of moduli spaces have
to be taken into account, namely the ones associated to the steps
$$
(\gamma_{\pm},\star)\glueto(\gamma_{\pm},\beta_{\star},\alpha_{\star})
\quad\text{ and }\quad
(\gamma_{+},\star)\glueto(\gamma_{-},\star).
$$
In both cases the bottom end of the configurations are free and the half
suspension of the relevant moduli space component is endowed with a
continuous evaluation map.

In the former however, one side is not associated to a broken trajectory
but to the constant one $\star$ and the half suspension is seen as having
4 sides. The evaluation map restricts to (see figure \ref{fig:SpecialSuspensions}) 
\begin{itemize}
\item 
  ${\beta_{\star}}_{|_{\R}}$ and ${\alpha_{\star}}_{|_{\R}}$ (suitably
  rescaled using the action) on the broken side
\item 
  the constant path $\{\star\}$ on the ``non broken'' side
\item
  the evaluation in $M$ of the Floer step
  $(\beta_{\star},\alpha_{\star})\glueto \star$ on the bottom.
\end{itemize}

In the latter, the half suspension can again be represented by a pentagon
and the evaluation map restricts to (see figure \ref{fig:SpecialSuspensions})
\begin{itemize}
\item 
  $\gamma_{+}$ and $\gamma_{-}$ on the upper left and right sides,
\item
  the constant trajectory $\star$ on the lower left and right sides,
\item
  the concatenation $\gamma_{+}\cdot\gamma_{-}$ on the bottom side.
\end{itemize}

The gluing construction used in \eqref{eq:CrocoDisc} adapts
straightforwardly to the $3$ special steps and results in a disc endowed
with a continuous evaluation map to $M$
\begin{equation}
  \label{eq:EvaluationOnCrocoHalfDisc}
  \Delta(W_{\star}(b))\xrightarrow{\eval}M.
\end{equation}
The restriction of the evaluation map to the boundary is the
concatenation of the trajectories $\gamma_{+}$ and $\gamma_{-}$ and of the
Floer loop formed by the lower steps used in the crocodile walk.

\section{Generation of the fundamental group}

In this section, homomorphisms from the group of Floer loops to
  that of Morse loops and vice versa are constructed, in order to prove
  theorem \ref{thm:EvalOntoPi1}, and to later study the relations.

In section \ref{sec:FromFloerToMorse}, we use the classical operation of
pushing arbitrary loops down by the flow in order to turn Floer loops into
Morse loops. This operation itself is not required for the proof of
theorem \ref{thm:EvalOntoPi1}, but it is reinterpreted purely in terms of
moduli spaces which makes it compatible with Floer theory. This is used
in section \ref{sec:FromMorseToFloer} to define a similar operation in
the reverse direction, turning Morse loops into Floer loops in the same
homotopy class. Finally, section \ref{sec:ProofEvalOntoPi1} gather the
proof of theorem \ref{thm:EvalOntoPi1}, which immediately follows from
the possibility of deforming Morse loops into (homotopical) Floer ones,
since the result is well known in the Morse setting.

\bigskip

Let $f$ be a Morse function having a single minimum at $\star$, and $g$ a
Riemannian metric on $M$ such that the pair $(f,g)$ is Morse Smale, and
all the relevant hybrid moduli spaces are cut out transversely.

Recall the Morse version of the definitions \ref{def:FloerLoopStep} and
\ref{def:FloerLoops}~: each choice of orientation on the unstable
manifold of each index $1$ Morse critical point defines a path we call a
Morse step (notice that since $f$ has a single minimum, all the steps are
in fact loops). Picking an arbitrary orientation for each such point $b$
allows to represent the associated Morse steps algebraically as
$b^{\pm}$, and hence to identify the group of Morse loops $\Gen{f}$ to
the free group generated by $\Crit_{1}(f)$. 

\subsection{From Floer to Morse loops}\label{sec:FromFloerToMorse}

\begin{lemma}\label{lem:PhiTop}
  There exist a group homomorphism $\Gen{H}\xrightarrow{\phi}\Gen{f}$ making
  the following diagram commutative~:
  \begin{equation}
    \label{eq:FloerToMorseDiagram}
    \xymatrix{
      \Gen{H} \ar[r]^{\eval}\ar[d]^{\phi} & \pi_{1}(M,\star)\ar@{=}^{\Id}[d]\\
      \Gen{f} \ar[r]^{\eval}             & \pi_{1}(M,\star)
    },
  \end{equation}
  i.e. such that 
  $$
  \forall w\in\Gen{H},\quad \eval(\phi(w))\sim\eval(w) \ \text{
    in }\ \pi_{1}(M,\star).
  $$  
\end{lemma}

\begin{proof}
  Pushing a generic topological loop $\gamma$ down by the flow of the
  Morse function $f$ deforms it into a Morse loop
  $\varphi_{f}^{+\infty}(\gamma)$, i.e. a word in the index $1$ critical
  points. Here generic means that the loop avoids the stable manifolds of
  all the index $k\geq2$ Morse critical points. Notice that the
  evaluation of the Floer steps form a finite collection of $1$-dimensional
  segments in $M$, and the stable manifolds of index $k\geq2$ critical
  points of $f$ are codimension $k\geq 2$ submanifolds. Therefore, for a
  generic (and even open dense) choice of $(f,g)$, Floer loops and such
  unstable manifolds do not meet, and we get a well defined  map 
  \begin{equation}
    \label{eq:HomotopyFlowDownBrut}
    \begin{array}{ccc}
      \GenBrut{H}&\xrightarrow{\phi}&\Gen{f}\\
      \gamma     &\mapsto           & \varphi^{+\infty}_{f}(\eval(\gamma)).
    \end{array}.  
  \end{equation}
  This map is obviously compatible both with the concatenation and
  cancellation rules, and hence induces a group homomorphism
  \begin{equation}
    \label{eq:HomotopyFlowDown}
    \begin{array}{ccc}
      \Gen{H}&\xrightarrow{\phi}&\Gen{f}
    \end{array}  
  \end{equation}
  Finally, $\phi$ is defined using a deformation and hence preserves the
  homotopy class, which means that the diagram
  \eqref{eq:HomotopyFlowDown} is commutative.
\end{proof}


Since the second row of \eqref{eq:FloerToMorseDiagram} is onto, theorem
\ref{thm:EvalOntoPi1} comes down to proving that any Morse loop can be
deformed into a Floer loop. Unfortunately, this deformation can not be
obtained like $\phi$ by pushing a loop down by a flow, since there is no
such thing as a Floer flow on the loop space.

However, a reinterpretation of $\phi$ in terms of moduli spaces and
crocodile walks can be given, allowing to generalize this definition to
the Floer setting and obtain a map in the
reverse direction. This reinterpretation is quickly sketched below, to
serve as an introduction for the reverse construction and to stress that
the two constructions are essentially the same, but will not be discussed
in details and could be skipped by the reader. The construction in the
reverse direction on the other hand, for which all the relevant technical
material was already introduced in section \ref{sec:HybridWalks}, will be
discussed in the next section.

\medskip

Consider a Floer loop step $\sigma=(\alpha,\beta,\beta',\alpha')$ through
some $y\in\CovOrbits_{1}(H)$. From our genericity assumption, the Morse
flow line $\gamma_{\alpha}$ (resp. $\gamma_{\alpha'}$) passing through
the center $\alpha(+\infty)$ (resp. $\alpha'(+\infty)$) of $\alpha$
(resp. $\alpha'$) ends at $\star$. Denote by $\bar{\alpha}$ (resp.
$\bar\alpha'$) the configuration obtained by appending to $\alpha$ (resp.
$\alpha'$) the piece of trajectory $\gamma_{\alpha}$ (resp.
$\gamma_{\alpha'}$) running from $\alpha(+\infty)$ (resp.
$\alpha'(+\infty)$) down to $\star$.

A crocodile walk can be run on the space of configurations consisting of
\begin{itemize}
\item 
  a trajectory from $y$ to some $x\in\Crit_{1}(f)\cup\CovOrbits_{0}(H)$,
\item 
  a trajectory from $x$ to $\star$, 
\item
  the (trivial~!) Morse trajectory $\star\in\MMorse(\star,\emptyset)$.
\end{itemize}
\begin{figure}[!ht]
  \centering
  \includegraphics[scale=.4]{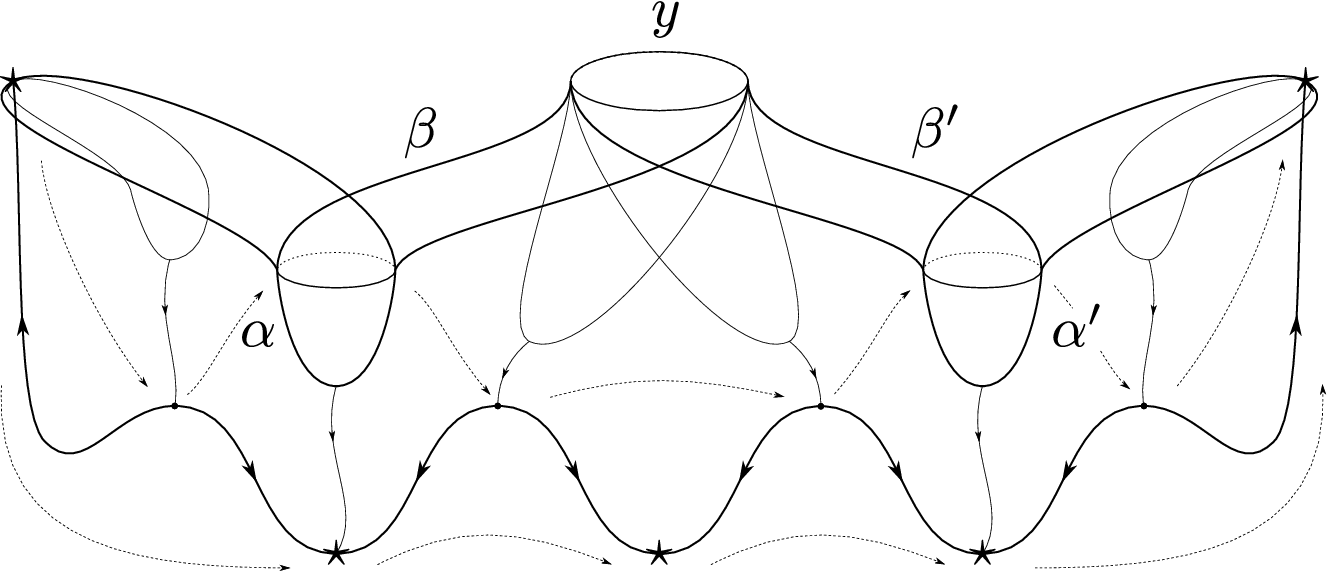}
  \caption{From Floer to Morse loops.}
  \label{fig:FloerToMorse}
\end{figure}
Starting with the configuration
$(\beta,\bar{\alpha},\star)$, the first upper step consists in gluing
$\beta$ and $\bar{\alpha}$. The other end of the associated component of
$\M(y,\star)$ is a configuration broken either at an index $0$
Hamiltonian orbit $x$, or at an index $1$ Morse critical point $b$  (see
figure \ref{fig:FloerToMorse}).

In the former case, the new configuration has to be
$(\beta',\bar\alpha',\star)$ (simply forget what happened to the Morse flow line 
and consider the definition of a Floer loop step).

In the latter, the lower part of the configuration is a Morse trajectory
$\gamma_{\pm}\in\M(b,\star)=\{\gamma_{-},\gamma_{+}\}$. The next (lower)
step consists in replacing $\gamma_{\pm}$ by $\gamma_{\mp}$ (recall from
the comments on definition \ref{def:FloerLoops} that this can be interpreted as
a step along the Morse moduli space $\MMorse(b,\emptyset)$). The next upper
step is then a gluing at $b$, and the same alternative holds again.

After a finite number of iterations of this process, the configuration
$(\beta',\bar\alpha',\star)$ has to be reached (from an upper step).
 Similarly to \eqref{eq:specialSteps}, moduli spaces involving interrupted
Morse trajectories give rise to the following special steps
\begin{gather*}
\sharp_{\bullet}(\beta',\bar{\alpha'},\star)=(\beta',\alpha')\\
\sharp^{\bullet}(\beta',\alpha')=(\beta,\alpha)\\
\sharp_{\bullet}(\beta,\alpha)=(\beta,\bar{\alpha},\star).
\end{gather*}
that close the walk orbit.

Let $W_{\sigma}$ be the orbit of the crocodile walk described above. The
lower non special steps in this orbit form a sequence of consecutive
Morse steps $\phi(\sigma)$.

Repeating this process for all the Floer loop steps $\sigma_{i}$ in a
Floer loop $\gamma=(\sigma_{1},\dots,\sigma_{N})$ (including the first
and last ones $\star\glueto(\beta_{\star},\alpha_{\star})$ and
$(\beta_{\star},\alpha_{\star})\glueto\star$ for which it still makes
sense), we get a sequence
$\phi(\gamma)=\phi(\sigma_{1})\dots\phi(\sigma_{N})$ which is a Morse
loop. This defines a map $\Gen{H}\to\Gen{f}$ which is a group homomorphism,
and it is a straightforward observation that this map is the same as
\eqref{eq:HomotopyFlowDown}.

Finally, observe that all the discs $\Delta(W_{\sigma_{i}})$ patch side
to side to form a disc endowed with an evaluation map realizing a
homotopy from $\eval(\gamma)$ to $\eval(\phi(\gamma))$.

\subsection{From Morse to Floer loops.}\label{sec:FromMorseToFloer}
\begin{lemma}\label{lem:PsiTop}
  There exist a group homomorphism $\Gen{f}\xrightarrow{\psi}\Gen{H}$ making
  the following diagram commutative~:
  \begin{equation}
    \label{eq:MorseToFloerDiagram}
    \xymatrix{
      \Gen{f} \ar[r]^{\eval}\ar[d]^{\psi} & \pi_{1}(M,\star)\ar@{=}^{\Id}[d]\\
      \Gen{H} \ar[r]^{\eval}             & \pi_{1}(M,\star)
    },
  \end{equation}
  i.e. such that 
  $$
  \forall w\in\Gen{f},\quad \eval(\psi(w))\sim\eval(w) \ \text{
    in }\ \pi_{1}(M,\star).
  $$  
\end{lemma}

\begin{proof}
  Let $b$ be an index $1$ critical point of $f$ and
  $(\gamma^{-}_{b},\gamma^{+}_{b})$ be the two Morse trajectories from
  $b$ to $\star$.

  Recall that the crocodile walk on the space
  $$
  B(b)=\bigcup_{y\in\CovOrbits_{1}(H)\cup\{\star\}}
  \M(b,y)\times\partial\M(y,\emptyset)
  $$
  was described in section \ref{sec:HybridWalks}. In particular, using
  the notations introduced there, it has a special orbit $W_{\star}(b)$
  (see figure \ref{fig:MorseToFloer}) of the from
  $$
  (\gamma_{b}^{-},\star),
  (\gamma_{b}^{-},\beta_{\star},\alpha_{\star}),
  (\gamma_{1},\beta_{1},\alpha_{1}), \dots
  (\gamma_{N},\beta_{N},\alpha_{N}),
  (\gamma_{b}^{+},\beta_{\star},\alpha_{\star}),%
  (\gamma_{b}^{+},\star).
  $$
  Recall from \eqref{eq:LowerCrocoSteps} that the lower steps in this
  orbit form a Floer loop. Denoting it by $\psi(b)$, we have 
  $$
  \psi(b)=\big((\star,\beta_{\star},\alpha_{\star}),
  \dots
  (\alpha_{2i-1},\beta_{2i-1},\beta_{2i},\alpha_{2i}),
  \dots
  (\alpha_{\star},\beta_{\star},\star)\big),
  $$
  and we get a map
  $$
  \GenBrut{f}\xrightarrow{\psi}\Gen{H},
  $$
  which is obviously compatible with both the concatenation and
  cancellation rules, and hence induces a group homomorphism
  \begin{equation}
    \label{eq:ImageOfGenerator}
    \Gen{f}\xrightarrow{\psi}\Gen{H}.
  \end{equation}

  Finally, the homotopy is provided by the disc $\Delta(W_{\star}(b))$
  and the evaluation map \eqref{eq:EvaluationOnCrocoHalfDisc}~: its
  restriction to the boundary is the concatenation of the Morse loop
  $\gamma^{-1}$ and the Floer loop $\psi(b)$.
\end{proof}

\subsection{Proof of theorem \ref{thm:EvalOntoPi1}.} 
\label{sec:ProofEvalOntoPi1}

\begin{proof}[Proof of theorem \ref{thm:EvalOntoPi1}]
  Theorem \ref{thm:EvalOntoPi1} is now a straightforward corollary of
  lemma \ref{lem:PsiTop}~: since the map on the first line of
  \eqref{eq:MorseToFloerDiagram} is onto, so has to be the map on the
  second.
\end{proof}

\section{Relations and fundamental groups}\label{sec:Relations}

It is natural to ask for a Floer theoretic interpretation of the
relations. It is the object of this section to provide a family of
generators of $\ker(\Gen{H}\xrightarrow{\eval}\pi_{1}(M,\star))$ that can
be expressed in terms of Floer and PSS moduli spaces.

\begin{remark}
  Although the subgroup of relations obviously only depends on
  $(H,J,\star,\cutoff)$, the proposed generators will depend on the
  choice of an additional auxiliary Morse function (and  metric). Being
  able to \emph{a priori} select  a finite family that would generate the
  relations and depend on $(H,J,\star,\cutoff)$ only would be more
  satisfactory but is unfortunately unclear.

  Moreover, resorting to a Morse function may seem to weaken the
  construction since Morse functions already give full access to the
  fundamental group. It should
    be observed however, that the Morse function is used in a different
  way from the usual one here~: it is used to define hybrid moduli
  spaces, mixing Morse and Floer objects, and the present description of
  the relations depicts how the Morse relations  have
  to be transported from the Morse to the Floer setting by some
  configurations of $1$-dimensional hybrid moduli spaces, and hence may
  gather some non trivial information.  
\end{remark}

\subsection{Floer-Morse-Floer relations}

Given a Floer loop $\gamma\in\Gen{H}$,  observe that the evaluations of
$\psi(\phi(\gamma))$ and $\gamma$ are homotopic (since both $\phi$ and
$\psi$ preserve the homotopy class), so that
$\gamma^{-1}\psi(\phi(\gamma))$ is always a relation.
\begin{definition}
  Define the set of ``Floer-Morse-Floer relations'' as
  $$
  \RelHomGen(H)=\{ \gamma^{-1}\psi(\phi(\gamma)),\ \gamma\in\Gen{H}\}.
  $$
\end{definition}

\begin{figure}[!ht]
  \centering
  \includegraphics[scale=.7]{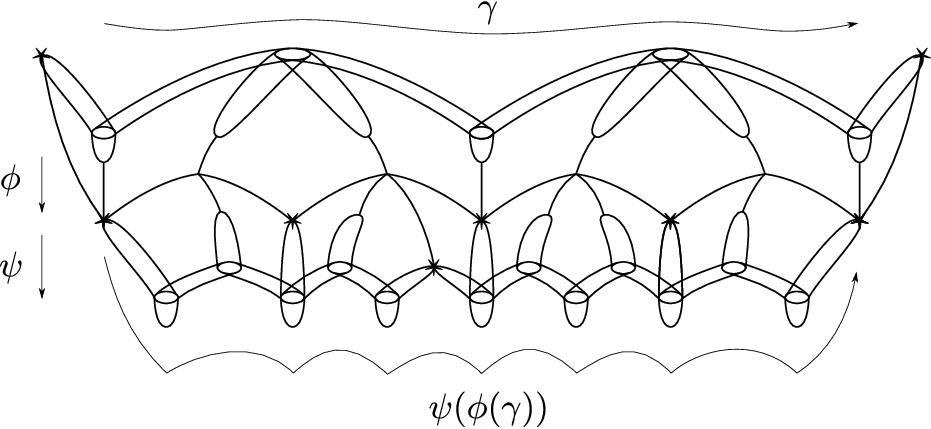}
  \caption{A relation in  $\RelHomGen(H)$.}
  \label{fig:HomotopyRelation}
\end{figure}

\begin{remark}
  The notation $\RelHomGen(H)$ only highlights the dependency on $H$ but
  this set depends in fact on all the auxiliary data
  $(H,J,f,g,\star, \cutoff)$.
\end{remark}

\begin{remark}
  The set $\RelHomGen(H)$ is not finite since there is one relation for
  each Floer loop. However, it is induced by the substitution rule at the
  Floer loop steps level
  $$
  \sigma \to \psi(\phi(\sigma))
  $$
  which is finite.
\end{remark}
Since $\phi$ and $\psi$ are described in terms of crocodile walk, so can
these relations. Glossing over the moduli spaces involving $\star$,
consider a Floer loop step $\sigma$ through some
$y_{0}\in\CovOrbits_{1}(H)$. The configurations consisting of 
\begin{itemize}
\item 
  a trajectory $\delta$ from $y_{0}$ to some
  $z\in\Crit_{1}(f)\cup\CovOrbits_{0}(H)$
\item
  a trajectory $\gamma$ from $z$ to some
  $y\in\{\star\}\cup\CovOrbits_{1}(H)$,
\item
  a trajectory $\beta$ from $y$ to some $x\in\CovOrbits_{0}(H)$,
\item
  a trajectory $\alpha\in\M(x,\emptyset)$
\end{itemize}
are broken three times and hence present $3$ levels where to perform a
gluing (or more generally a step). The relation is obtained by running
the crocodile walk on the two lower gluings ``from $\star$ to $\star$'',
then performing one upper gluing, and repeating this process.

\subsection{Relations associated to Morse $2$-cells}

Given an index $2$ Morse critical point $c$ of $f$, let $\rho_{c}$ be the
relation in $\Gen{f}$ given by the boundary of the associated $2$ cell and define~:
\begin{gather}
  \label{eq:2CellsMorseRelations}
  \RelMorseGen(f) = \{\rho_{c},\ c\in\Crit_{2}(f)\}
  \\
  \label{eq:2CellsFloerRelations}
  \RelMorseGen(H) = \{\psi(\rho_{c}),\ c\in\Crit_{2}(f)\}
\end{gather}

\begin{remark}
  The notation $\RelMorseGen(H)$ only highlights the dependency on $H$ but
  this set depends in fact on all the auxiliary
  data $(H,J,f,g,\star,\cutoff)$.
\end{remark}

\begin{figure}[!ht]
  \centering
  \includegraphics[scale=.4]{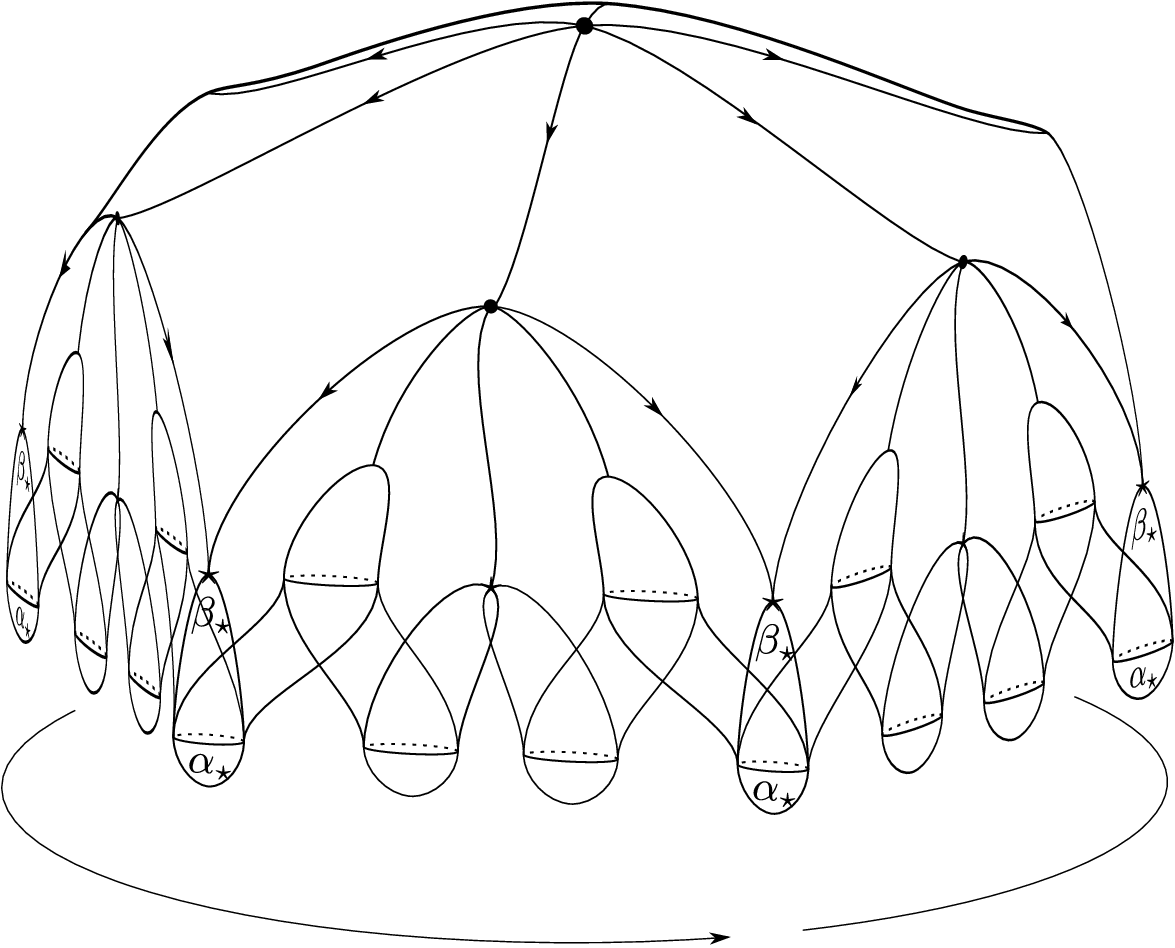}
  \caption{A Floer relation associated to a Morse $2$-cell.}
  \label{fig:Morse2CellRelation}
\end{figure}

\begin{remark}
  For all $\rho\in\RelMorseGen(H)$ we have $\eval(\rho)= 1$ in
  $\pi_{1}(M,\star)$, so that 
    $\RelMorseGen(H)$ is indeed a collection of relations.
\end{remark}
\begin{remark}
  These relations can also be described in terms of crocodile walks.
  Glossing over the moduli spaces involving $\star$ again, consider an
  index $2$ Morse critical point $c$ and the configurations consisting of
  \begin{itemize}
  \item 
    a trajectory $\delta$ from $c$ to some $z\in\Crit_{1}(f)$,
  \item 
    a trajectory $\gamma$ from $z$ to some
    $y\in\CovOrbits_{1}(H)\cup\Crit_{0}(f)$,
  \item 
    a trajectory $\beta$ from $y$ to some $x\in\CovOrbits_{0}(H)$,
  \item 
    a trajectory $\alpha\in\M(x,\emptyset)$.
  \end{itemize}
  The relation associated to $c$ can be obtained using the same algorithm
  as discussed previously, i.e. running the crocodile walk on the two
  lower levels ``from $\star$ to $\star$'', then performing one upper
  step, and repeating this process.
\end{remark}

\subsection{Fundamental group}
\bigskip
  
We can finally define the subgroup of relations~:
\begin{definition}
Denote by $\Rel{H}$ the normal subgroup of $\Gen{H}$ generated by
$\RelHomGen(H)$ and $\RelMorseGen(H)$~:
$$
\Rel{H} = <\RelHomGen(H),\RelMorseGen(H)>.
$$
\end{definition}
\begin{remark}
The group $\Rel{H}$ obviously depends on $(H,J,\star,\cutoff)$,
but it is a consequence of theorem \ref{thm:IsomorphismToPi1} that it
does not depend on $(f,g)$.
\end{remark}

\begin{definition}
  The Floer fundamental group associated to $(H,J,\star,\cutoff)$ is defined
  as the group
  $$
  \pi_{1}(H,\star)=\Gen{H}/\Rel{H}.
  $$
\end{definition}

\begin{remark}
  The group should be denoted as $\pi_{1}(H, J, \star, \cutoff)$
  to emphasize its dependency on all the auxiliary data but it is kept
  implicit to reduce notations.
\end{remark}

In the same way, let $\Rel{f}=<\RelMorseGen(f)>$ be the normal subgroup of $\Gen{f}$
generated by the boundary of Morse $2$-cells,
$\pi_{1}(f,\star):=\Gen{f}/\Rel{f}$ and recall the well known fact that
$\pi_{1}(f,\star)\simeq\pi_{1}(M,\star)$.

\begin{theorem}\label{thm:IsomorphismToPi1}
  The evaluation induces a group isomorphism
  $$
  \xymatrix{%
    \pi_{1}(H,\star)\ar[r]^{\eval}&\pi_{1}(M,\star).%
  }
  $$
  The maps $\phi$ and $\psi$ also induce isomorphisms which are
  inverse one of the other~:  
  $$
  \xymatrix{%
    \pi_{1}(H,\star)\ar@<2pt>[r]^{\phi}&\ar@<2pt>[l]^{\psi}\pi_{1}(f,\star). }
  $$
\end{theorem}

\begin{proof}

  \begin{enumerate}
  \item 
    Compatibility with the relations for $\psi$~:\\
    observe that $\Rel{f}=<\RelMorseGen(f)>$ and
    $\psi(<\RelMorseGen(f)>)\ \subset\ <\psi(\RelMorseGen(f))>$. Since
    $\psi(\RelMorseGen(f))=\RelMorseGen(H)$ and
    $<\RelMorseGen(H)>\subset\Rel{H}$, we have 
    $$
    \psi(\Rel{f})\subset \Rel{H}.
    $$

  \item
    Compatibility with the relations for $\phi$~:\\
    Similarly,
    $\phi(\Rel{H})\subset<\phi(\RelHomGen(H))\cup\phi(\RelMorseGen(H))>$.
    But for $\rho\in\RelHomGen(H))\cup\RelMorseGen(H)$, we have
    $\eval(\phi(\rho))=\eval(\rho)=1$, so that
    $$
    \phi(\Rel{H})\subset\Rel{f}.
    $$

  \item $\phi\circ\psi=\Id_{\pi_{1}(f,\star)}$~:\\
    This follows directly from $\eval\circ\phi\circ\psi=\eval$. In
    particular this implies surjectivity of $\phi$ and injectivity of $\psi$.

  \item $\psi\circ\phi=\Id_{\pi_{1}(H,\star)}$~:\\
    This is built in the definition of the relations $\RelHomGen(H)$~:
    for $\gamma\in\Gen{H}$, we have $\gamma^{-1}\psi(\phi(\gamma))\in
    \RelHomGen(H)$, so that $\gamma = \psi(\phi(\gamma))$ in
    $\pi_{1}(H,\star)$. This implies injectivity of $\phi$ and
    surjectivity of $\psi$.

  \item 
    $\ker\big(\eval:\Gen{H}\to\pi_{1}(M,\star)\big)=\Rel{H}$~:\\
    The relation $\Rel{H}\subset\ker\eval$ is obvious since this is true
    for all the generators of $\Rel{H}$. Conversely, let
    $\gamma\in\Gen{H}$ such that $\eval(\gamma)=1$ in $\pi_{1}(M,\star)$.
    Then $\eval(\phi(\gamma))=1$ so that $\phi(\gamma)\in \Rel{f}$. As a
    consequence 
    $$
    \psi(\phi(\gamma))\in\psi(\Rel{f})\subset\Rel{H}.
    $$
    Finally, since $\gamma^{-1}\psi(\phi(\gamma))\in\Rel{H}$, we have
    $\gamma\in\Rel{H}$.

    This ends the proof that
    $\pi_{1}(H,\star)\xrightarrow{\eval}\pi_{1}(M,\star)$ is injective,
    and hence an isomorphism since it was already proven to be surjective.
  \end{enumerate}
\end{proof}

\section{Application and proof of theorem \ref{thm:AtLeastOneOrbit}.}

\subsection{Generating $\pi_{1}(M)$ with steps}
\label{sec:ProofEnoughSteps} The theorem \ref{thm:EnoughSteps} is a
direct consequence of a weaker version of theorem \ref{thm:EvalOntoPi1}
where Floer loops are replaced by Floer steps.

\begin{proof}[Proof of theorem \ref{thm:EnoughSteps}]
  Fix a generic set of data $(H,J,\star,\cutoff,f,g)$ where $\star$
  is the single minimum of the Morse function $f$. Let $\sigma_{\star}$
  denote the special step $\star\glueto(\beta_{\star},\alpha_{\star})$.
  Let $\Steps(H)$ be the free group generated by all the Floer loop steps
  but the special one.
 
  Recall that the map $\phi$ was defined at the step level~: 
  \begin{equation}
    \label{eq:PhiOnSteps}
    \xymatrix{
      \Steps(H)\ar[r]^{\phi}&\Gen{f}\ar[r]^{\eval}&\pi_{1}(M,\star)\\
    }
  \end{equation}
  (notice that although Floer loop steps evaluate as free paths in $M$
  and not necessarily as based loops, they are still pushed down into
  Morse based loops by $\phi$ because the Morse function was chosen to
  have only one index $0$ critical point).

  Notice that the left hand side of \eqref{eq:MainEstimate} is nothing
  but the number of generators of $\Steps(H)$, so that theorem
  \ref{thm:EnoughSteps} reduces to proving that in \eqref{eq:PhiOnSteps},
  $\eval\circ\phi$ is onto.

  Observe now that in a loop $w\in\GenBrut{H}$, the only occurrences of
  $\sigma_{\star}$ and $\sigma_{\star}^{-1}$ are~:
  \begin{itemize}
  \item 
    $\sigma_{\star}$ at the beginning of $w$,
  \item 
    $\sigma_{\star}^{-1}$ at the end of $w$,
  \item 
    possible pairs $(\sigma_{\star}^{-1}\sigma_{\star})$ within $w$.
  \end{itemize}

  In particular, this means that removing $\sigma_{\star}$ and
  $\sigma_{\star}^{-1}$ at the ends of the loops defines an injective group
  homomorphism
  $$
  \xymatrix{
    \Gen{H}\ar@{^(->}[r]^{\tau} & \Steps(H)
  }.
  $$

  We end up with the following commutative diagram~: 
  \begin{equation}
    \label{eq:StepsDiagram}
    \xymatrix{
      \Gen{H}\ar[r]^{\phi}\ar@{^(->}[d]^{\tau} & \Gen{f}\ar[d]^{\tau'}\ar[r]^{\eval}&
      \pi_{1}(M,\star)\ar[d]^{\tau''}\\
      \Steps(H)\ar[r]^{\phi}&\Gen{f}\ar[r]^{\eval}&\pi_{1}(M,\star)\\
    }
  \end{equation}
  where $\tau'$ and $\tau''$ are the conjugation by
  $\phi(\sigma_{\star}^{-1})$ and $\eval(\phi(\sigma_{\star}^{-1}))$
  respectively. In particular, surjectivity of the composition of the
  maps appearing on the first row implies that of the second.
\end{proof}

\subsection{Proof of theorem
  \ref{thm:AtLeastOneOrbit}.}\label{sec:ProofAtLeastOneOrbit}

In this section, we want to prove theorem \ref{thm:AtLeastOneOrbit},
namely that if $\pi_{1}(M)\neq\{1\}$, then every non-degenerate
Hamiltonian $H$ should have at least one contractible $1$-periodic orbit
of index $1$ (i.e. Conley-Zehnder index $1-n$) with non vanishing
multiplicity. 

This is not a consequence of the above construction, but uses similar
ideas arranged slightly differently~: it is based on a variant of the
crocodile walk to patch (suspensions) of $1$-dimensional PSS moduli
spaces together and fill any Morse loop with a disc when there are no
index $1$ Hamiltonian orbit. 

\medskip

Let $H$ be a non degenerate Hamiltonian, and pick a triple $(J,f,g)$
where $J$ is a (possibly time dependent) almost complex structure
compatible with $\omega$, $f$ a Morse function with a single minimum
denoted by $\star$ and $g$ a Riemannian metric such that
$(H,J,\star,\cutoff,f,g)$ satisfies our transversality
assumptions. We pick coherent orientations on all the $0$ and
$1$-dimensional moduli spaces $\M(b,x)$ for $b\in\Crit_{0,1}(f)$ and
$x\in\CovOrbits_{0}(H)\cup\Crit_{0}(f)$.

Suppose $H$ has no index $1$ orbit, or more precisely that it has no
index $1$ orbit with non vanishing multiplicity~: this means there are no
Floer trajectories from an index $1$ to an index $0$ orbit that admits at
least one augmentation. For convenience, let 
$$
\CovOrbits_{0}(H)^{*}=\{x\in\CovOrbits_{0}(H), \M(x,\emptyset)\neq\emptyset\}.
$$ 
Our assumption can then be written as~:
$$
\forall y\in\CovOrbits_{1}(H), \forall x\in\CovOrbits_{0}(H)^{*},
\ \ \M(y,x)=\emptyset.
$$

Let $b$ be an index $1$ Morse critical point, such that the unstable
manifold of $b$ defines a non trivial loop $\gamma$ in $M$, and let
$\gamma_{-}$ and $\gamma_{+}$ be the two Morse flow lines rooted at $b$.
For convenience, we consider $\gamma$ as based at $b$ and let~:
$$
\gamma = \gamma_{+}\cdot\gamma_{-}^{-1}.
$$

For $x\in\CovOrbits_{0}(H)^{*}$ consider the space 
$$
B(b) = \{\gamma_{-},\gamma_{+}\}\times\M(\star,x)
$$
Since $H$ has no index $1$ orbit related to $x$ by a Floer trajectory,
$B(b)$ is the set of all broken hybrid trajectories from $b$ to $x$.

In particular, gluing $\gamma_{\pm}$ with a trajectory
$\beta\in\M(\star,x)$ defines a $1$-dimensional family of trajectories
from $b$ to $x$ whose other end has to be of the same form. This defines
a one to one correspondence~:
$$
\begin{array}[b]{cccl}
  B(b)&\xrightarrow{\sigma} &B(b)\\
(\gamma_{\epsilon},\beta)&\mapsto&(\gamma_{\epsilon'},\beta')
&\text{ such that }(\gamma_{\epsilon},\beta)\glueto(\gamma_{\epsilon'},\beta')
\end{array}.
$$
Permuting $\gamma_{-}$ and $\gamma_{+}$ defines another one to one
correspondence 
$$
\begin{array}[b]{cccl}
  B(b)&\xrightarrow{\tau} &B(b)\\
(\gamma_{\pm},\beta)&\mapsto&(\gamma_{\mp},\beta)
\end{array}.
$$
Notice both $\sigma$ and $\tau$ reverse the orientation.

Consider now an orbit of $\rho=\tau\circ\sigma$. It has to be cyclic, and
 is a sequence 
$$
(\gamma_{\epsilon_{1}},\beta_{1}),\dots,(\gamma_{\epsilon_{k}},\beta_{k}),
$$
(with $\epsilon_{i}=\pm1$) such that
$(\gamma_{\epsilon_{i}},\beta_{i})\glueto(\gamma_{-\epsilon_{i+1}},\beta_{i+1})$, with the convention that
$(\gamma_{\epsilon_{k+1}},\beta_{k+1})=(\gamma_{\epsilon_{1}},\beta_{1})$).

To each gluing, is associated a $1$-dimensional space, and we let
$\Sigma_{i}$ be its suspension. It is a diamond, endowed with an
evaluation map to $M$ that coincides with
\begin{itemize}
\item 
$\gamma_{\epsilon_{i}}$ on the  upper left edge,
\item 
$\beta_{i}$ on the  lower left edge,
\item 
$\gamma_{-\epsilon_{i+1}}$ on the upper right edge,
\item 
$\beta_{i+1}$ on the lower right edge.
\end{itemize}

Gluing all these diamonds side by side along the lower edges provides a disc,
endowed with a continuous evaluation map to $M$, whose restriction to the
boundary is
$$
\gamma_{\epsilon_{1}}^{-1}\gamma_{-\epsilon_{2}}\gamma_{\epsilon_{2}}^{-1}\dots
\gamma_{-\epsilon_{k}}\gamma_{\epsilon_{k}}^{-1}\gamma_{-\epsilon_{1}}.
$$

This loop is therefore trivial, but
$\gamma_{-\epsilon_{i}}\gamma_{\epsilon_{i}}^{-1} =
\gamma^{-\epsilon_{i}}$ so $\gamma^{\sum\epsilon_{i}}=1$. Moreover, the
orientation of the couple $(\gamma_{\epsilon_{i}},\beta_{i})$ is constant
with respect to $i$ (because one moves from one to the next by two
gluings and the orientation is reversed by each gluing) and it can be
supposed to be positive without loss of generality. This means that
$\epsilon_{i}=\orient(\beta_{i})$ for all $i$ and hence
$\sum\epsilon_{i}=\sum\orient(\beta_{i})$ (where $\orient(\beta_{i})$ is
the orientation of $\beta_{i}$). As a consequence, we get 
$$
 \gamma^{\sum \orient(\beta_{i})}\sim 1 \text{ in }\pi_{1}(M,\star).
$$

Observe now that the orbits of $\rho$ induce a partition of
$\M(\star,x)$, so repeating this for all the orbits $O_{1},\dots,O_{N}$
of $\rho$, we derive
$$
\gamma^{\sum_{O_{1}} \orient(\beta_{i})}\dots\gamma^{\sum_{O_{N}} \orient(\beta_{i})} 
=\gamma^{\sum_{\beta\in\M(\star,x)}\orient(\beta)}
=\gamma^{n_{x}}\sim 1
\text{ in }\pi_{1}(M,\star),
$$ 
where $n_{x}=\Cardalg\M(\star,x)$ is the algebraic number of elements in
$\M(\star,x)$ (i.e. the sum of signs $\pm1$ associated to each
  element in $\M(\star,x)$ according to a choice of coherent
  orientations). Recall this number is the component along $x$ of the
image of $\star$ under the PSS homomorphism from the Morse to the Floer
complex (using $\Z$ coefficients)~:
$$
PSS_{MF}(\star) = \sum_{x\in\CovOrbits_{0}(H)}n_{x}\, x.
$$

Let $PSS_{FM}$ be the PSS homomorphism from the Floer to the Morse complex.
Since $PSS_{FM}\circ PSS_{MF}$ induces the identity in homology, we have 
$$
\sum_{x\in\CovOrbits_{0}(H)}n_{x}\, m_{x} =1,
$$
where $m_{x}=\Cardalg(\M(x,\emptyset))$. In particular we also have
$\sum_{x\in\CovOrbits_{0}(H)^{*}}n_{x}\, m_{x} =1  $.

As a consequence we have 
$$
\gamma=\gamma^{\sum_{x\in\CovOrbits_{0}(H)^{*}} n_{x}m_{x}}= 1 \text{ in }\pi_{1}(M,\star).
$$
This is a contradiction, since we supposed $\gamma$ was non trivial. This
ends the proof of theorem \ref{thm:AtLeastOneOrbit}.

\section{Stable Morse version}\label{sec:StableMorse}

To some extent, a stable Morse function can be considered as a simplified
finite dimensional model for the action functional on the free loop
space. This section is devoted to a quick sketch of the analogue of the
main construction in the stable Morse setting. Although it would deserve
a dedicated discussion, it is only addressed here to shed some light on
the phenomena encountered along the construction that do not appear in
the usual Morse setting, like the existence of several steps through the
same critical point or of steps through $\star$. Therefore, we limit
ourselves to the defining the relevant moduli spaces, and
leave all the proofs to the reader.

\subsection{Setting}
Let $M$ be a smooth closed manifold of dimension $n$, $\star$ a point in
$M$, $N_{\pm}$ be two integers, $N=N_{+}+N_{-}$ and $H$ a Morse function
on $M\times\R^{N}$ that is quadratic at infinity with signature
$(N_{+},N_{-})$. Namely, we suppose that there is a compact set $K$ such
that $\forall(m,u,v)\in (M\times\R^{N_{+}}\times\R^{N_{-}})\setminus K,
H(m,u,v)=\Vert u \Vert^{2}-\Vert v\Vert^{2}$.

For convenience, the Morse index will be shifted by $N_{-}$ and we let,
for a critical point $x$ of $H$~:
$$
|x|=\mu(x)-N_{-}
$$
where $\mu$ denotes the usual Morse index.

We also pick a Riemannian metric $g$ on $M\times\R^{N}$ and denote by
$\phi^{t}$ the associated negative gradient flow of $H$.

\subsection{Moduli spaces}
For $x,y\in\Crit(H)$ the usual space of trajectories from $y$ to $x$ can
be described as
\begin{align*}
  \M(y,x)&=\big(W^{u}(y)\cap W^{s}(x)\big)_{/\R}.
\end{align*}
The counterpart of the ``augmentations'' required for the
construction are now trajectories ``hitting $M\times\R^{N_{+}}$'', namely
\begin{align*}
  \M(y,\emptyset)&=W^{u}(y)\cap \big(M\times\R^{N_{+}}\times\{0\}\big),
\end{align*}
and the counterpart of the evaluation map $u\mapsto u(+\infty)$ is the
projection $M\times\R^{N_{+}}\xrightarrow{\pi_{+}} M$~:
$$
\begin{array}{ccl}
  \M(y,\emptyset)&\xrightarrow{\eval}&M\\
  (m,u,0)&\mapsto&m
\end{array}.
$$

Similarly, the spaces
\begin{align*}
  \M(\star,x)&=\big(\{\star\}\times\R^{N_{-}}\big)\cap W^{s}(x),\\
  \M(\star,\emptyset)&=\{(p,R)\in\{\star\}\times\R^{N_{-}}\times ]0,+\infty[,
\phi^{R}(p)\in M\times\R^{N_{+}} \}
\end{align*}
are the counterparts of the spaces that were denoted by the same notations in
the Floer setting. 

\begin{figure}[!ht]
  \centering
  \includegraphics[scale=.4]{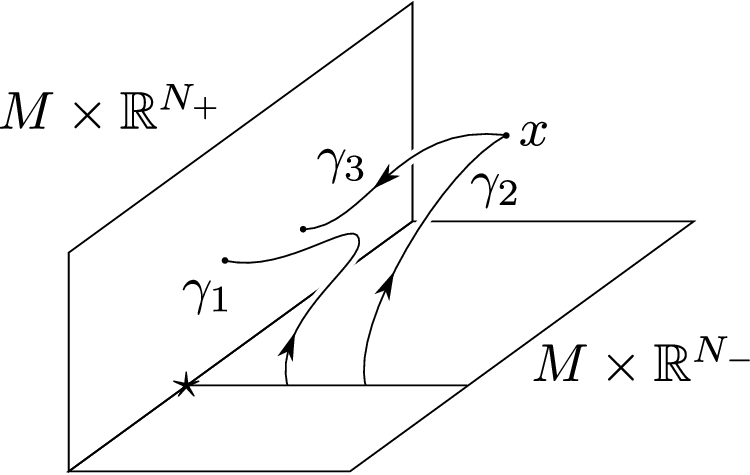} %
  \caption{Augmentations and other configurations in the stable 
    Morse setting~:
    $\gamma_{1}\in\M(\star,\emptyset)$,
    $\gamma_{2}\in\M(\star,x)$,    
    $\gamma_{3}\in\M(x,\emptyset)$.}
\label{fig:StableMorse}
\end{figure}

The triple $(H,g,\star)$ is supposed to be chosen generically so that all
the considered moduli spaces are cut out transversely. In this situation,
they are all smooth manifolds of dimension~:
\begin{align*}
  \dim\M(y,x)&=|y|-|x|-1\\
  \dim\M(y,\emptyset)&=|y|\\
  \dim\M(\star,x)&=-|x|\\
  \dim\M(\star,\emptyset)&=1
\end{align*}

Moreover, they are compact up to breaking at intermediate critical points
(although $M\times\R^{N}$ is not compact), and the gluing construction also
makes sense in this setting. 

Notice in particular that $\M(\star,\emptyset)$ still has a projection
$\pi$ to $[0,+\infty]$, and that $\pi^{-1}(\{0\})$ consists of exactly
one point, namely $\star$ itself, since $\pi^{-1}(\{0\}) =
(\{\star\}\times\R^{N_{-}})\cap(M\times\R^{N_{+}}) = \{(\star,0)\}$.

With these notations, the definitions given in the Floer setting make
sense literally and give rise to suitable notions of ``stable Morse steps and
loops'' and to the associated group $\Gen{H}$.

\bigskip

Picking now a Morse function $f$ on $M$ having a single minimum at $\star$
and a metric on $M$, one can consider the following hybrid moduli spaces
(see figure \ref{fig:MorseStableMorse})~:
\begin{align*}
\M(b,x)&=\pi_{-}^{-1}(W^{u}_{f}(b))\cap W^{s}(x),\\
\M(x,b)&=W^{u}(x)\cap \pi_{+}^{-1}(W^{s}_{f}(b)),
\end{align*}
where $W^{u}_{f}$ and $W^{s}_{f}$ denote the stable and unstable
manifolds with respect to the negative gradient of $f$ in $M$, and
$\pi_{\pm}$ are the projections
$M\times\R^{N_{\pm}}\xrightarrow{\pi_{\pm}}M$.

\begin{figure}[!ht]
  \centering
  \includegraphics[scale=.4]{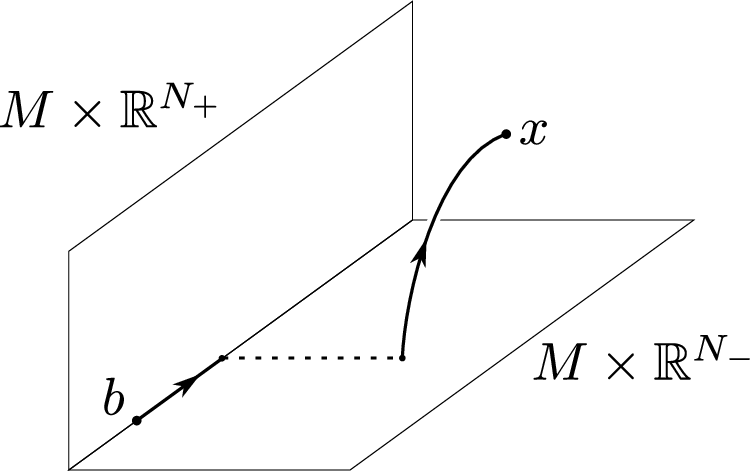} %
  \includegraphics[scale=.4]{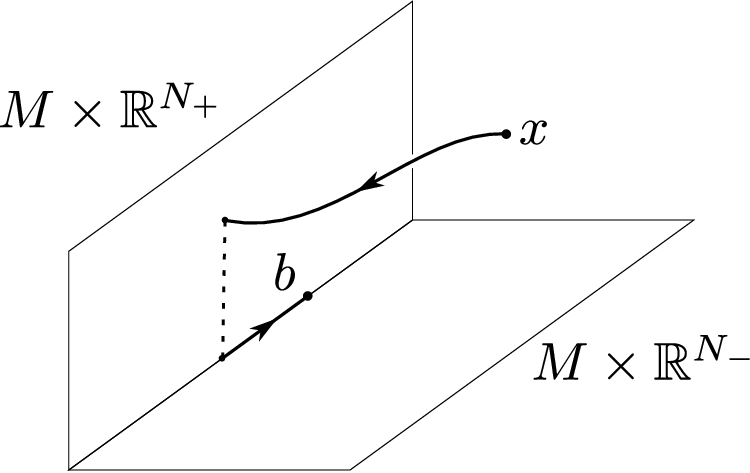} %
  \caption{Hybrid trajectories between Morse and ``stable Morse'' 
    critical points.}
\label{fig:MorseStableMorse}
\end{figure}

Using these hybrid moduli spaces, the proof of the following statement follows
literally that of its Floer analogue and is left to the reader~:
\begin{theorem}
The map $\Gen{H}\xrightarrow{\eval}\pi_{1}(M,\star)$ is onto.   
\end{theorem}

\subsection{Multiplicities}

Since the stable Morse situation is much easier to handle than the Floer
one, it is now not hard to give examples where several steps are
associated to the same index $1$ critical point or where there is more
than one step going through $\star$. The figure \ref{fig:StableMorseStep}
illustrates the former phenomenon, and the latter is similar.

\begin{figure}[!ht]
  \centering
  \includegraphics[scale=.6]{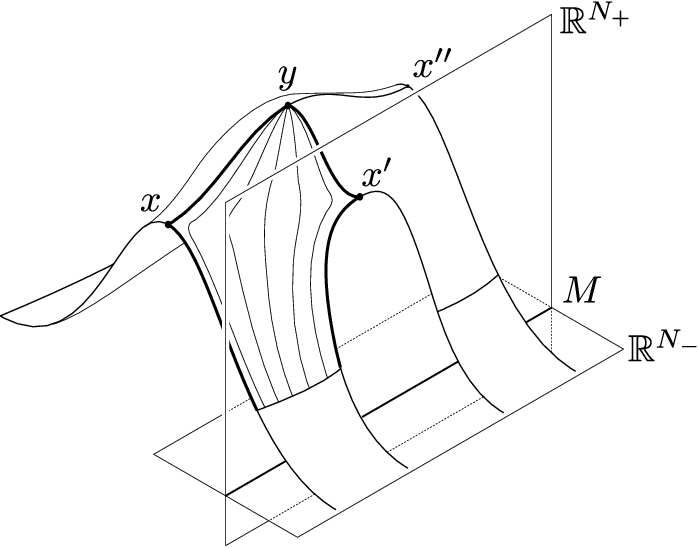} %
  \caption{Two steps through the same 
    index $1$ critical point $y$~: one involves $x$ and $x'$, the 
    other $x'$ and $x''$.}
\label{fig:StableMorseStep}
\end{figure}


\end{document}